\documentclass{amsart}
\usepackage{amssymb}
\usepackage{color}
\usepackage{graphicx}

\newcommand{\Fractal[1]}{\maltese_{#1}}

\newcommand{\Fixx[1]}{*_{#1}}
\newcommand{\w}{\omega}
\newcommand{\setmap}{\multimap}

\newcommand{\la}{\langle}
\newcommand{\ra}{\rangle}
\newcommand{\IN}{\mathbb N}
\newcommand{\Lip}{\mathrm{Lip}}

\newcommand{\e}{\varepsilon}

\newcommand{\Comp}{\mathrm{Comp}}
\newcommand{\IR}{\mathbb R}
\newcommand{\IZ}{\mathbb Z}
\newcommand{\IC}{\mathbb C}

\newcommand{\RP}{{\mathbb R}\mathsf{P}}

\newtheorem{theorem}{Theorem}
\newtheorem{lemma}{Lemma}

\newtheorem{proposition}{Proposition}
\newtheorem{problem}{Problem}
\theoremstyle{definition}
\newtheorem{definition}{Definition}
\newtheorem{remark}{Remark}
\newtheorem{example}{Example}

\title[Micro and Macro Fractals]{Micro and Macro Fractals generated\\ by multi-valued dynamical systems}
\author{Taras Banakh and Natalia Novosad}
\address{T.Banakh: Ivan Franko National University of Lviv, Ukraine, and Jan Kochanowski University in Kielce, Poland}
\email{t.o.banakh@gmail.com}
\address{N.Novosad: Institute for Applied Problems of Mechanics and Mathematics, Naukova 3b, Lviv, Ukraine}
\email{natalia.kasper@gmail.com}

\begin{document}
\begin{abstract} Given a multi-valued function $\Phi:X\setmap X$ on a topological space $X$ we study the properties of its {\em fixed fractal\/} $\Fractal[\Phi]$, which is defined as the closure of the orbit $\Phi^\w(\Fixx[\Phi])=\bigcup_{n\in\w}\Phi^n(\Fixx[\Phi])$ of the set $\Fixx[\Phi]=\{x\in X:x\in\Phi(x)\}$ of fixed points of $\Phi$. A special attention is paid to the duality between {\em micro-fractals} and {\em macro-fractals}, which are fixed fractals $\Fractal[\Phi]$ and $\Fractal[\Phi^{-1}]$ for a contracting compact-valued function $\Phi:X\setmap X$ on a complete metric space $X$. With help of algorithms (described in this paper) we generate various images of macro-fractals which are dual to some well-known micro-fractals like the  fractal cross, the Sierpi\'nski triangle, Sierpi\'nski carpet, the Koch curve, or the fractal snowflakes. The obtained images show that macro-fractals have a large-scale fractal structure, which becomes clearly visible after a suitable zooming.
\end{abstract}

\maketitle

\section{Introduction}

In this paper we generalize the classical theory of deterministic fractals and for each multi-valued function $\Phi:X\setmap X$ on a topological space $X$ we define its fixed fractal $\Fractal[\Phi]\subset X$ as the closure $\bar\Phi^\w(\Fixx[\Phi])$ of the orbit $\Phi^\w(\Fixx[\Phi])=\bigcup_{n\in\w}\Phi^n(\Fixx[\Phi])$ of the set $\Fixx[\Phi]=\{x\in X:x\in\Phi(x)\}$ of fixed points of $\Phi$. This definition of a fractal agrees with the classical definition of a deterministic fractal because for a system of contracting functions $f_1,\dots,f_m:X\to X$ on a complete metric space $X$ the fractal $\Fractal[\Phi]$ of the multi-valued function $\Phi:x\mapsto\{f_1(x),\dots,f_m(x)\}$ coincides with the attractor of the IFS $\{f_1,\dots,f_m\}$ defined in the standard way, see \cite{Barnsley}.

By a {\em multi-valued function} (briefly, a {\em multi-function}) $\Phi:X\setmap X$ on a topological space $X$ we understand any subset $\Phi\subset X\times X$. For a subset $A\subset X$ by $\Phi(A)=\{y\in X:\exists x\in A\;(x,y)\in\Phi\}$ we denote its image under $\Phi$. We put $\Phi^0(A)=A$ and $\Phi^{n+1}(A)=\Phi(\Phi^n(A))$ for $n\ge 0$. The set $\Phi^\w(A)=\bigcup_{n\in\w}\Phi^n(A)$ is called the {\em orbit} of $A$. The closure $\bar\Phi^\w(\Fixx[\Phi])$ of the orbit $\Phi^\w(\Fixx[\Phi])$ of the set $\Fixx[\Phi]=\{x\in X:x\in\Phi(x)\}$ is called the {\em fixed fractal} of $\Phi$ and is denoted by $\Fractal[\Phi]$. A pair $(X,\Phi)$ is called a {\em multi-valued dynamical system} \cite{Akin}.

Fixed fractals of multi-valued dynamical systems are central objects of our study. We start with the following  three problems related to fixed fractals.

\begin{problem} Study the interplay between properties of a multi-valued function $\Phi:X\setmap X$ and properties of its fixed fractal $\Fractal[\Phi]$.
\end{problem}

For each multi-valued function $\Phi:X\setmap X$ on a topological space $X$ we can consider its inverse multi-valued function $$\Phi^{-1}:X\setmap X,\;\;\;\Phi^{-1}:y\mapsto\{x\in X:y\in\Phi(x)\},$$ and the corresponding fixed fractal $\Fractal[\Phi^{-1}]$. In such way, we obtain a dual pair of fixed fractals $\Fractal[\Phi]$ and $\Fractal[\Phi^{-1}]$. Observe that both these fractals are closures of orbits of the same set $\Fixx[\Phi]=\Fixx[\Phi^{-1}]$ of fixed points.

\begin{problem} Given a multi-valued function $\Phi:X\setmap X$ on a topological space $X$, study  the interplay between properties of the fixed fractal $\Fractal[\Phi]$ and its dual fixed fractal $\Fractal[\Phi^{-1}]$.
\end{problem}

There are many algorithms that allow us to see fractals of various sorts.

\begin{problem} Elaborate effective algorithms for visualizing fixed fractals $\Fractal[\Phi]$ of multi-valued functions \mbox{$\Phi:X\setmap X$}  defined on simple spaces $X$ (like the plane).
\end{problem}

In fact, our initial motivation was to study the duality between {\em micro-fractals}, i.e., fixed fractals $\Fractal[\Phi]$ of contracting compact-valued functions $\Phi:X\setmap X$ on complete metric spaces $X$ and their dual fixed fractals $\Fractal[\Phi^{-1}]$ called {\em macro-fractals}. Such fractals will be considered in Sections~\ref{s:micro} and \ref{s:macro}. In Section~\ref{s:algo} we describe some algorithms of drawing fixed fractals. The obtained images indicate that macro-fractals have a non-trivial fractal structure seen on a macro-scale, in contrast to micro-fractals whose fractal structure can be seen only on the micro-scale.

\section{Some general facts on multi-valued functions and their fixed fractals}

We have defined the fixed fractal $\Fractal[\Phi]$ of a multi-valued function $\Phi:X\setmap X$ on a topological space $X$ as the closure $\bar\Phi^\w(\Fixx[\Phi])$ of the orbit $\Phi^\w(\Fixx[\Phi])=\bigcup_{n\in\w}\Phi^n(\Fixx[\Phi])$ of the set $\Fixx[\Phi]=\{x\in X:x\in\Phi(x)\}$ of fixed points of $\Phi$.

It follows that $\Phi^0(\Fixx[\Phi])=\Fixx[\Phi]\subset\Phi(\Fixx[\Phi])$ and by induction, $\Phi^n(\Fixx[\Phi])\subset\Phi^{n+1}(\Fixx[\Phi])$ for every $n\in\w$. This means that the sequence
of sets $\big(\Phi^n(\Fixx[\Phi])\big)_{n\in\w}$ is increasing. We shall prove that this sequence converges to the fixed fractal $\Fractal[\Phi]$ in the Vietoris topology on the power-set $2^X$.

The {\em Vietoris topology} on $2^X$ is generated by the sub-base consisting of the sets $$\la U\ra^+=\{A\in 2^X:A\cap U\ne\emptyset\}\mbox{ \ and \ }\la F\ra^-=\{A\in 2^X:A\cap F=\emptyset\}$$where $U$ and $F$ run over open and closed subsets of $X$, respectively.

\begin{proposition}\label{p:gen} For any multi-valued function $\Phi:X\setmap X$ on a topological space $X$ its fixed fractal $\Fractal[\Phi]$ is the limit of the increasing sequence of sets $\big(\Phi^n(\Fixx[\Phi])\big)_{n\in\w}$ in the Vietoris topology on $2^X$.
\end{proposition}

\begin{proof} Given an open neighborhood ${\mathcal U}\subset 2^X$ of the fixed fractal $\Fractal[\Phi]=\bar\Phi^\w(\Fixx[\Phi])$ in the Vietoris topology of $2^X$, we need to find $N\in\IN$ such that $\Phi^n(\Fixx[\Phi])\in{\mathcal U}$ for all $n\ge N$. We lose no generality assuming that the neighborhood ${\mathcal U}$ is of the basic form:
$${\mathcal U}=\la X\setminus U_0\ra^-\cap \la U_1\ra^{+}\cap\dots\cap\la U_m\ra^+$$
for some non-empty open sets $U_0,U_1,\dots,U_m$. Since $\bar\Phi^\w(\Fixx[\Phi])\in{\mathcal U}$, for every $k\le m$ the closure $\bar\Phi^\w(\Fixx[\Phi])$ of the orbit $\Phi^\w(\Fixx[\Phi])$ meets the open set $U_k$. Then $\Phi^\w(\Fixx[\Phi])\cap U_k\ne\emptyset$ and $\Phi^{n_k}(\Fixx[\Phi])\cap U_k\ne\emptyset$ for some number $n_k\in\w$. Then for the number $N=\max\{n_1,\dots,n_m\}$ and each number $n\ge N$ the set $\Phi^n(\Fixx[\Phi])$ contains the union $\bigcup_{k=1}^m\Phi^{n_k}(\Fixx[\Phi])$ and hence meets each set $U_k$, $1\le k\le m$. Since $\Phi^n(\Fixx[\Phi])\subset\bar\Phi^\w(\Fixx[\Phi])\subset U_0$, we conclude that
$$\Phi^n(\Fixx[\Phi])\in\la X\setminus U_0\ra^-\cap\la U_1\ra^+\cap\dots\cap \la U_m\ra^+={\mathcal U}.$$
\end{proof}

\begin{remark} Proposition~\ref{p:gen} suggests a simple algorithm of drawing the fixed fractal $\Fractal[\Phi]$ of a multi-valued function $\Phi:X\setmap X$: {\em choose a sufficiently large $n$ and draw the set $\Phi^n(\Fixx[\Phi])$}. This set can be considered as an approximation of the fractal $\Fractal[\Phi]$.
\end{remark}

A characteristic feature of classical deterministic fractals is their self-similarity (see \cite{Barnsley}, \cite{Crownover}), which can be equivalently defined as the invariance of the fractal under the action of the multi-valued function.

A subset $A$ of a space $X$ will be called {\em $\Phi$-invariant} under the action of a multi-valued function $\Phi:X\setmap X$ if $\Phi(A)=A$. We shall show that the for each bi-continuous multi-valued function $\Phi:X\setmap X$ its fixed fractal $\Fractal[\Phi]$ is $\Phi$-invariant, so  self-similar in some sense.

We recall that a multi-valued function $\Phi:X\setmap X$ on a topological space $X$ is
\begin{itemize}
\item {\em lower semicontinuous} if for each open subset $U\subset X$ the set $\Phi^{-1}(U)=\{x\in X:\Phi(x)\cap U\ne\emptyset\}$ is open in $X$;
\item {\em upper semicontinuous} if for every closed subset $F \subset X$ the set $\Phi^{-1}(F)=\{x\in X:\Phi(x)\cap F\ne\emptyset\}$ is closed in $X$;
\item {\em continuous} if $\Phi$ is lower and upper semicontinuous.
\item {\em bi-continuous} if $\Phi$ and $\Phi^{-1}$ are continuous.
\end{itemize}

It is easy to see that a multi-valued function $\Phi:X\setmap X$ is continuous if and only if it is continuous as a single-valued function $\Phi:X\to 2^X$ into the power-set $2^X$ endowed with the Vietoris topology. In the following lemma by $\bar{A}$ we denote the closure of a subset $A\subset X$ in a topological space $X$.

\begin{lemma}\label{l:cont} Let $\Phi:X\setmap X$ be a multi-valued function on a topological space $X$, and $A\subset X$.
\begin{enumerate}
\item If $\Phi$ is lower semicontinuous, then $\Phi(\bar A)\subset\overline{\Phi(A)}$;
\item If $\Phi^{-1}$ is upper semicontinuous, then $\Phi(\bar A)\supset\overline{\Phi(A)}$;
\item If $\Phi$ is bi-continuous, then $\Phi(\bar A)=\overline{\Phi(A)}$.
\end{enumerate}
\end{lemma}

\begin{proof} 1. Assume that $\Phi$ is lower semicontinuous but $\Phi(\bar A)\not\subset\overline{\Phi(A)}$. Consider the open subset $U=X\setminus\overline{\Phi(A)}$ of $X$.  By the lower semicontinuity of $\Phi$, it inverse image $\Phi^{-1}(U)$ is open and intersects the set $\bar A$. Then we can find a point $a\in A\cap \Phi^{-1}(U)$. For this point $a$ we get $\Phi(a)\subset\Phi(A)\subset \overline{\Phi(A)}=X\setminus U$, which is not possible as $a\in\Phi^{-1}(U)$ and hence $\Phi(a)\cap U\ne \emptyset$.
\smallskip

2. Assume that $\Phi^{-1}$ is upper semicontinuous but  $\Phi(\bar A)\not\supset\overline{\Phi(A)}$. Then some point $y\in\overline{\Phi(A)}$ does not belong to $\Phi(\bar A)$ and hence $\Phi^{-1}(y)\cap \bar A=\emptyset$. Now consider the open subset $U=X\setminus\bar A$, which contains the set $\Phi^{-1}(y)$. Since the multi-function $\Phi^{-1}$ is upper semi-continuous, the set $V=\{z\in X:\Phi^{-1}(z)\subset U\}$ is open in $X$. Since $y\in V\cap\overline{\Phi(A)}$, there is a point $a\in A$ such that $\Phi(a)\cap V$ contains some point $b$. For this point we get $\Phi^{-1}(b)\subset U$. Then $a\in\Phi^{-1}(b)\cap A\subset U\cap A=\emptyset$, which is a desired contradiction.
\smallskip

3. The third item trivially follows from the first two items.
\end{proof}

Since $$\Phi(\Phi^\w\big(\Fixx[\Phi])\big)=\Phi\big(\bigcup_{n\in\w}\Phi^n(\Fixx[\Phi])\big)=\bigcup_{n\in\w}\Phi^{n+1}(\Fixx[\Phi])=\Phi^\w(\Fixx[\Phi]),$$ Lemma~\ref{l:cont} implies the following proposition:

\begin{proposition} Let $\Phi:X\multimap X$ be a multi-valued function on a topological space $X$.
\begin{enumerate}
\item  If $\Phi$ is lower semicontinuous, then $\Phi(\Fractal[\Phi])\subset\Fractal[\Phi]$.
\item If $\Phi^{-1}$ is upper semicontinuous, then $\Phi(\Fractal[\Phi])\supset\Fractal[\Phi]$.
\item If $\Phi$ is bi-continuous, then $\Phi(\Fractal[\Phi])=\Fractal[\Phi]$.
\end{enumerate}
\end{proposition}

\section{Fixed fractals of contracting multi-functions on complete metric spaces}\label{s:micro}

In this section we consider the fixed fractals of contracting multi-valued functions on complete metric spaces. Such fractals will be called {\em micro-fractals}. The obtained results can be considered as a multi-valued generalization of the classical theory of deterministic fractals \cite{Barnsley},  (in spirit of  Ethan Akin \cite{Akin}, Jan Andres, Ji\v si Fi\v ser, and Miroslav Rypka \cite{AF}, \cite{AR}  who also used multi-valued functions in the theory of multi-valued dynamical systems and multivalued fractals).

Let $(X,d)$ be a metric space. To measure the distance between subsets $A,B\subset X$ we use the Hausdorff distance $$d_H(A,B)=\max\{\sup_{a\in A}d(a,B),\sup_{b\in B}d(b,A)\}\in[0,\infty]$$where $d(a,B)=\inf_{b\in B}d(a,b)$. This formula is well-defined only for non-empty subsets $A,B\subset X$. To extend the definition of the Hausdorff distance to the whole power-set $2^X$ we put $d_H(\emptyset,\emptyset)=0$ and $d_H(A,\emptyset)=d_H(\emptyset,A)=\infty$ for any non-empty subset $A\subset X$.

A multi-valued function $\Phi:X\setmap X$ on a metric space $(X,d)$ is called {\em Lipschitz} if it has finite {\em Lipschitz constant}
$$\Lip(\Phi)=\sup_{x\ne y}\frac{d_H\big(\Phi(x),\Phi(y)\big)}{d(x,y)}.$$ If $\Lip(\Phi)<1$, then $\Phi$ is called {\em contracting}.

It is easy to see that for contracting multi-functions, $\Phi_1,\dots,\Phi_n:X\setmap X$ (seen as subsets of $X\times X$) their union $\Phi=\Phi_1\cup\dots\cup\Phi_n$ is a contracting multi-function with Lipschitz constant $$\Lip(\Phi)\le\max\{\Lip(\Phi_1),\dots,\Lip(\Phi_n)\}.$$

A multi-valued function $\Phi:X\setmap X$ is called {\em compact-valued} (resp. {\em finite-valued\/}) if for each $x\in X$ the set $\Phi(x)\subset X$ is compact (finite) and not empty.

It is easy to see that each contracting compact-valued function $\Phi:X\setmap X$ of a metric space is continuous (as a single-valued function $\Phi:X\to 2^X$ into the power-set $2^X$ endowed with the Vietoris topology).

\begin{theorem}\label{t:micro} Let $\Phi:X\setmap X$ be a contracting compact-valued function on a complete metric space $X$ and let $\lambda=\Lip(\Phi)<1$.
\begin{enumerate}
\item The set $\Fixx[\Phi]$ of fixed points of $\Phi$ is compact and not empty.
\item If $\Phi$ is finite-valued, then the set $\Fixx[\Phi]$ is finite.
\item The micro-fractal $\Fractal[\Phi]=\bar\Phi^\w(\Fixx[\Phi])$ is compact and not empty.
\item $\Phi(\Fractal[\Phi])=\Fractal[\Phi]$.
\item $d_H\big(\Fractal[\Phi],\Phi^n(B)\big)\le \frac{\lambda^n}{1-\lambda}\,d_H(\Phi(B),B)$ for every $n\in\w$ and a non-empty compact subset $B\subset X$.
\end{enumerate}
\end{theorem}

\begin{proof} It is well-known that the space $\Comp(X)\subset 2^X$ of non-empty compact subsets of $X$, endowed with the Hausdorff metric $d_H$ is a complete metric space, see e.g. \cite[2.7]{Barnsley}. The contracting compact-valued function $\Phi:X\setmap X$ can be seen as a contracting single-valued function $\Phi:X\to \Comp(X)$. Being contracting, this map is continuous. Consequently, for any non-empty compact subset $K\subset X$ its image $\{\Phi(x):x\in K\}$ is a compact subset of $\Comp(X)$ and then its union $\Phi(K)=\bigcup_{x\in K}\Phi(x)$ is a compact subset of $X$, see \cite[5.1]{Wicks}. This means that the single-valued function $$\Phi:\Comp(X)\to\Comp(X),\;\;\;\Phi:K\mapsto\Phi(K),$$ is well-defined. We claim that this function $\Phi$ is contracting and its Lipschitz constant is equal to the Lipschitz constant $\lambda<1$ of the function $\Phi:X\to\Comp(X)$.

We need to check that $d_H(\Phi(A),\Phi(B))\le\lambda\cdot d_H(A,B)$ for any compact sets $A,B\in\Comp(X)$. Given any point $x\in\Phi(A)=\bigcup_{a\in A}\Phi(a)$, find a point $a\in A$ with $x\in\Phi(a)$. By the compactness of $B$, there is a point $b\in B$ with $d(a,b)=d(a,B)$.
Then $$d(x,\Phi(B))\le d(x,\Phi(b))\le d_H(\Phi(a),\Phi(b))\le\lambda\cdot d(a,b)=\lambda\cdot d(a,B)\le \lambda\cdot d_H(A,B).$$
By analogy, for each point $y\in \Phi(B)$, we can prove that $d(y,\Phi(A))\le\lambda\cdot d_H(A,B)$. This implies that $$d_H(\Phi(A),\Phi(B))=\max\{\max_{x\in \Phi(A)}d(x,\Phi(B)),\max_{y\in\Phi(B)}d(y,\Phi(A))\}\le\lambda\cdot d_H(A,B).$$

Therefore, the function $\Phi:\Comp(X)\to\Comp(X)$ is contracting with Lipschitz constant $\le\lambda<1$.
Now Banach Contracting Principle \cite[\S3.3]{Crownover} implies that the contracting function $\Phi:\Comp(X)\to\Comp(X)$ has a unique fixed point $K\in 2^X$ which can be found as the limit of the sequence $\big(\Phi^n(B)\big)_{n\in\w}$ that starts with any compact set $B\in \Comp(X)$. Moreover, this sequence converges to its limit $K$ with the velocity
$$d_H(K,\Phi^n(B))\le\frac{\lambda^n}{1-\lambda}d_H(\Phi(B),B),\;\;n\in\w.$$

Now we are ready to prove the statements (1)--(5) of Theorem~\ref{t:micro}.
\smallskip

First we prove that the set $\Fixx[\Phi]$ is not empty. This will follow as soon as we check that $$d(x_0,\Fixx[\Phi])\le \frac1{1-\lambda}d(x_0,\Phi(x_0))$$for every point $x_0\in X$. Indeed, by the compactness of $\Phi(x_0)$, there is a point $x_1\in\Phi(x_0)$ with $d(x_0,x_1)=d(x_0,\Phi(x_0))$. Proceeding by induction, for every $n\in\w$, we can find a point $x_{n+1}\in\Phi(x_n)$ with $d(x_n,x_{n+1})=d(x_n,\Phi(x_n))$. We claim that so-defined sequence $(x_n)_{n\in\w}$ is Cauchy.
Indeed, for every $n\in\IN$ we get
$$d(x_{n},x_{n+1})=d(x_{n},\Phi(x_{n}))\le d_H(\Phi(x_{n-1}),\Phi(x_{n}))\le\lambda\cdot d(x_{n-1},x_{n})$$ and by induction, $$d(x_{n},x_{n+1})\le \lambda^n d(x_0,x_1)=\lambda^n d(x_0,\Phi(x_0)).$$Then for any $n<m$ we get
$$d(x_n,x_m)\le\sum_{i=n}^{m-1}d(x_i,x_{i+1})\le\sum_{i=n}^{m-1}\lambda^i d(x_0,\Phi(x_0))<\frac{\lambda^n}{1-\lambda}d(x_0,\Phi(x_0)),$$
which witnesses that the sequence $(x_n)_{n\in\w}$ is Cauchy and by the completeness of $X$, converges to some point $x_\infty\in X$ such that $d(x_n,x_\infty)\le\frac{\lambda^n}{1-\lambda}d(x_0,\Phi(x_0))$ for every $n\in\w$. The continuity of the function $\Phi:X\to\Comp(X)$ guarantees that the sequence $\big(\Phi(x_n)\big)_{n\in\w}$ converges to $\Phi(x_\infty)$ in $\Comp(X)$ and
$$x_\infty=\lim_{n\to\infty}x_{n+1}\in\lim_{n\to\infty}\Phi(x_n)=\Phi(\lim_{n\to\infty}x_n)=\Phi(x_\infty),$$
which means that $x_\infty$ is a fixed point of the multi-function $\Phi$ and hence $\Fixx[\Phi]\ni x_\infty$ is not empty. Moreover,
$$d(x_0,\Fixx[\Phi])\le d(x_0,x_\infty)\le\frac1{1-\lambda}d(x_0,\Phi(x_0)).$$
\smallskip

The continuity of the function $\Phi:X\to\Comp(X)$ implies that the set $\Fixx[\Phi]=\{x\in X:x\in\Phi(x)\}$ is closed in $X$. The compactness of $\Fixx[\Phi]$ will follow as soon as we check that $\Fixx[\Phi]\subset K$. Fix any point $x\in\Fixx[\Phi]$ and consider the increasing sequence $\big(\Phi^n(x)\big)_{n\in\w}$ of non-empty compact subsets of $X$. By the Banach Contraction Principle, this sequence tends to the unique fixed point $K$ of the contracting function $\Phi:\Comp(X)\to\Comp(X)$. This implies that the orbit $\Phi^\w(x)$ is dense in $K$.
Consequently, the orbit $\Phi^\w(\Fixx[\Phi])\supset\Phi^\w(x)$ of the set $\Fixx[\Phi]$ also is dense in $K$, which implies that $\Fixx[\Phi]\subset\bar\Phi^\w(\Fixx[\Phi])=K$. This completes the proof of the first statement and simultaneously proves the third, fourth and fifth statements.

It remains to prove that the set $\Fixx[\Phi]$ is finite if the function $\Phi$ is finite-valued.
If the compact set $\Fixx[\Phi]$ is infinite, then it contains a non-isolated point $x\in\Fixx[\Phi]$. Since the set $\Phi(x)$ is finite, there is $\e>0$ such that $d(x,z)>\e$ for each $z\in\Phi(x)\setminus \{x\}$. Now choose any point $y\in \Fixx[\Phi]$ with $0<d(y,x)<\e/2$. Since $d_H(\Phi(x),\Phi(y))\le\lambda\cdot d(y,x)$, for the point $y\in\Phi(y)$ there is a point $z\in \Phi(x)$ with $$d(y,z)=d(y,\Phi(x))\le d_H(\Phi(y),\Phi(x))\le\lambda\cdot d(y,x)<\e/2.$$
The point $z$ belongs to the set $\Phi(x)$ and lies on the distance
$$d(x,z)\le d(x,y)+d(y,z)\le\tfrac\e2+\tfrac\e2=\e$$from $x$. Now the choice of $\e$ guarantees that $z=x$ and hence $0<d(y,x)=d(y,z)\le\lambda\cdot d(y,x)<d(y,x)$, which is a desirable contradiction.
\end{proof}

\begin{remark} Theorem~\ref{t:micro}(5) suggests a (well-known) algorithm for drawing the micro-fractal $\Fractal[\Phi]=\bar\Phi^\w(\Fixx[\Phi])$ of $\Phi$ with a given precision $\e>0$: Take any non-empty compact subset $B\subset X$, calculate the Hausdorff distance $d_H(B,\Phi(B))$ and find a number $n\in\IN$ such that $\frac{\lambda^n}{1-\lambda}d_H(\Phi(B),B)<\e$. Then draw the set $\Phi^n(B)$. By Theorem~\ref{t:micro}(5), the set $\Phi^n(B)$ approximates the micro-fractal $\Fractal[\Phi]$ with precision  $d_H(\Phi^n(B),\Fractal[\Phi])<\e$.
\end{remark}

\section{Fixed fractals of $*$-repelling multi-functions on topological spaces}\label{s:macro}

In this section we define a class of multi-functions whose fixed fractals are called macro-fractals, and like micro-fractals, can be drawn by certain efficient algorithms.

\begin{definition} A multi-valued function $\Phi:X\setmap X$ on a topological space $X$ is defined to be {\em $*$-repelling} if the  family $$\{\Phi^{n+1}(\Fixx[\Phi])\setminus\Phi^{n}(\Fixx[\Phi])\}_{n\in\w}$$ is locally finite in $X$ in the sense that each point $x\in X$ has a neighborhood $O_x\subset X$ that intersects only finitely many sets $\Phi^{n+1}(\Fixx[\Phi])\setminus\Phi^{n}(\Fixx[\Phi])$, $n\in\w$.
\end{definition}

Fixed fractals of $*$-repelling multi-valued functions will be called {\em macro-fractals}.
The structure of macro-fractals is described in the following simple proposition that can be easily derived from the definition of the $*$-repelling property.

\begin{proposition}\label{p1:rep} Let $\Phi:X\setmap X$ be a $*$-repelling multi-function on a topological space $X$ such that for every $n\in\w$ the set $\Phi^n(\Fixx[\Phi])$ is closed in $X$. Then
\begin{enumerate}
\item the fixed fractal $\Fractal[\Phi]$ coincides with the orbit $\Phi^\w(\Fixx[\Phi])$ of the set $\Fixx[\Phi]$ and hence $\Phi(\Fractal[\Phi])=\Fractal[\Phi]$;
\item for any compact subset $E\subset X$ there is a number $n\in\w$ such that $E\cap\Fractal[\Phi]=E\cap\Phi^n(\Fixx[\Phi])$.
\end{enumerate}
\end{proposition}

For a finite-valued function $\Phi:X\setmap X$ on a $T_1$-space $X$ the $*$-repelling property of $\Phi$ is equivalent to the discrete property of its fixed fractal $\Fractal[\Phi]$. Let us recall that a topological space $X$ is called a {\em $T_1$-space} if each finite subset of $X$ is closed in $X$.

\begin{proposition}\label{p2:rep} A finite-valued function $\Phi:X\setmap X$ on a $T_1$-space $X$ with finite set $\Fixx[\Phi]$ is $*$-repelling if and only if its fixed fractal $\Fractal[\Phi]$ is discrete.
\end{proposition}

\begin{proof} The ``if'' part is trivial. To prove the ``only if'' part, assume that the finite-valued function $\Phi$ is $*$-repelling.

Taking into account that the set $\Fixx[\Phi]=\Phi^0(\Fixx[\Phi])$ is finite, by induction, we can prove that for every $n\in\w$ the set $\Phi^n(\Fixx[\Phi])$ is finite. By the $*$-repelling property of $\Phi$, the family of finite sets $$\{\Phi^{n+1}(\Fixx[\Phi])\setminus\Phi^{n}(\Fixx[\Phi])\}_{n\in\w}$$ is locally finite in $X$ and consequently, its union $$\Phi^\w(\Fixx[\Phi])=\bigcup_{n\in\w}\Phi^{n+1}(\Fixx[\Phi])\setminus \Phi^n(\Fixx[\Phi])$$is a closed discrete subset of $X$. Then the fixed fractal $\Fractal[\Phi]=\bar\Phi^\w(\Fixx[\Phi])=\Phi^\w(\Fixx[\Phi])$ also is discrete.
\end{proof}

\begin{remark} Proposition~\ref{p1:rep} suggests a simple algorithm for drawing the fixed fractal of a $*$-repelling multi-valued function: {\em given a compact subset $E\subset X$} (which can be thought as a computer screen showing us the  piece $E\cap\Fractal[\Phi]$ of the fixed fractal), {\em find $n\in\w$ such that  $E\cap \Fractal[\Phi]=E\cap\Phi^n(\Fixx[\Phi])$ and draw the set} $E\cap\Phi^n(\Fixx[\Phi])$ (which coincides with the piece $E\cap\Fractal[\Phi]$ of the fixed fractal $\Fractal[\Phi]$).

For expanding multi-functions $\Phi$ (introduced below) the number $n$ can be calculated effectively.
\end{remark}

A multi-valued function $\Phi:X\setmap X$ on a complete metric space $X$ is defined to be {\em expanding} if its inverse multi-function $\Phi^{-1}:X\setmap X$ is compact-valued and contracting.

\begin{theorem}\label{t:macro} An expanding multi-valued function $\Phi:X\setmap X$ on a complete metric space $X$ is $*$-repelling if and only if for some $k\in\IN$ the number
$$\delta=\inf\{d(x,y):x\in \Phi^{k}(\Fixx[\Phi])\setminus\Phi^{k-1}(\Fixx[\Phi]),\;\;y\in\Fractal[\Phi^{-1}]\}$$is strictly positive. In this case for every $n\in\w$ and the neighborhood $$E_n=\Big\{x\in X:d(x,\Fractal[\Phi^{-1}])\le\frac{\delta}{\Lip(\Phi^{-1})^n}\Big\}$$ of the micro-fractal $\Fractal[\Phi^{-1}]$
we get $E_n\cap\Phi^\w(\Fixx[\Phi])=E_n\cap\Phi^{k+n}(\Fixx[\Phi])$.
\end{theorem}

\begin{proof} Since the multi-function $\Phi$ is expanding, its inverse $\Phi^{-1}$ is contracting and has Lipschitz constant $\lambda=\Lip(\Phi^{-1})<1$.
\smallskip

To prove the ``only if'' part, assume that the expanding multi-function $\Phi:X\setmap X$ is $*$-repelling.  By Theorem~\ref{t:micro}, the micro-fractal $\Fractal[\Phi^{-1}]$ of the contracting multi-function $\Phi^{-1}$ is compact. The $*$-repelling property of $\Phi$ implies that the family $\big\{\Phi^{k+1}(\Fixx[\Phi])\setminus\Phi^k(\Fixx[\Phi])\big\}_{k\in\w}$ is locally finite. Consequently, the compact set $\Fractal[\Phi^{-1}]$ has a neighborhood $U\subset X$ that meets only finitely many sets $\Phi^{k+1}(\Fixx[\Phi])\setminus\Phi^k(\Fixx[\Phi])$, $k\in\w$. Then for some $k\in\IN$ the sets $\Phi^{k+1}(\Fixx[\Phi])\setminus\Phi^k(\Fixx[\Phi])$ and $U$ are disjoint and the number
$$\delta=\inf\{d(x,y):x\in \Phi^{k}(\Fixx[\Phi])\setminus\Phi^{k-1}(\Fixx[\Phi]),\;\;y\in\Fractal[\Phi^{-1}]\}\ge d(\Fractal[\Phi^{-1}],X\setminus U)$$
is strictly positive by the compactness of the micro-fractal $\Fractal[\Phi^{-1}]$.
This completes the proof of the ``only if'' part of the theorem.
\smallskip

To prove the ``if'' part, assume that for some  $k\in\IN$ the
number
$$\delta=\inf\{d(x,y):x\in \Phi^{k}(\Fixx[\Phi])\setminus\Phi^{k-1}(\Fixx[\Phi]),\;\;y\in\Fractal[\Phi^{-1}]\}$$is strictly positive. The $*$-repelling property of $\Phi$ will follow as soon as we check that for every $n\in\w$ and the neighborhood $$E_n=\big\{x\in X:d(x,\Fractal[\Phi^{-1}])\le\delta/\lambda^n\big\}$$ of the micro-fractal $\Fractal[\Phi^{-1}]$
we get $E_n\cap\Phi^\w(\Fixx[\Phi])=E_n\cap\Phi^{k+n}(\Fixx[\Phi])$.

Assume conversely that the intersection $E_n\cap\Phi^\w(\Fixx[\Phi])$ contains some point $x\notin
E_n\cap\Phi^{k+n}(\Fixx[\Phi])$. Then $x\in\Phi^m(\Fixx[\Phi])\setminus\Phi^{m-1}(\Fixx[\Phi])$ for some $m>k+n$. Let $x_m=x$ and by reverse induction, for every non-negative $i<m$ choose a point $x_{i}\in \Phi^i(\Fixx[\Phi])$ such that  $x_{i+1}\in \Phi(x_i)$. Taking into account that $x_m\notin\Phi^{m-1}(\Fixx[\Phi])$, we conclude that $x_{m-1}\notin \Phi^{m-2}(\Fixx[\Phi])$ and by induction, $x_i\notin\Phi^{i-1}(\Fixx[\Phi])$ for all positive $i\le m$.

By the compactness of the micro-fractal $\Fractal[\Phi^{-1}]$, for every $i\le m$ there is a point $y_i\in\Fractal[\Phi^{-1}]$ such that $d(x_i,y_i)=d(x_i,\Fractal[\Phi^{-1}])$. It follows that $$\Phi^{-1}(y_i)\subset \Phi^{-1}(\Fractal[\Phi^{-1}])=\Fractal[\Phi^{-1}].$$

Taking into account that the function $\Phi^{-1}:X\to\Comp(X)$ is contracting with contracting constant $\lambda$,  we conclude that
$d(x_{i-1},y_{i-1})=d(x_{i-1},\Fractal[\Phi^{-1}])\le d(x_{i-1},\Phi^{-1}(y_i))\le d_H(\Phi^{-1}(x_i),\Phi^{-1}(y_i))\le\lambda\cdot d(x_{i},y_{i})$. Using this inequality, by induction we can prove  that
$$d(x_i,y_i)\le \lambda^{m-i}d(x_m,y_m)=\lambda^{m-i}d(x_m,\Fractal[\Phi^{-1}])\le\lambda^{m-i}\delta/\lambda^n$$
for every $i\le m$. In particular, for $i=k$ we get
$$d(x_k,\Fractal[\Phi^{-1}])\le d(x_k,y_k)\le\delta\cdot\lambda^{m-k-n}\le \lambda\cdot\delta<\delta$$which contradicts the definition of the number $\delta$ because $x_k\in\Phi^k(\Fixx[\Phi])\setminus\Phi^{k-1}(\Fixx[\Phi])$.
\end{proof}

\section{A simple example of a dual pair of micro and macro fractals}

In this section we consider a simple example of a multi-valued function $\Phi:\IR\setmap \IR$ for which the fixed fractals $\Fractal[\Phi]$ and $\Fractal[\Phi^{-1}]$ can be found explicitly in analytic form.

The function $\Phi:\IR\setmap\IR$ assigns to each point $x\in\IR$ the doubleton $\Phi(x)=\{3x,3x-2\}$. The inverse map $\Phi^{-1}$ is given by the formula $\Phi^{-1}(y)=\{\frac13y,\frac13y+\frac23\}$. The sets $\Fixx[\Phi]=\Fixx[\Phi^{-1}]$ of the fixed points of $\Phi$ and $\Phi^{-1}$ coincide with the doubleton $\{0,1\}$.

The multi-function $\Phi^{-1}$ is contracting being the union $\Phi^{-1}=f_0\cup f_1$ of two contracting
functions $f_1:x\mapsto \frac13x$ and $f_2:x\mapsto \frac13x+\frac23$, and has Lipschitz constant  $\Lip(\Phi)=\frac13$.

\begin{proposition} The fixed fractal $\Fractal[\Phi^{-1}]$ coincides with the Cantor set
$$\mu=\Big\{\sum_{k=1}^\infty 2x_k3^{-k}:(x_k)_{k=1}^\infty\in\{0,1\}^\IN\Big\}.$$
\end{proposition}

\begin{proof} By (the proof) of Theorem~\ref{t:micro}, the micro-fractal $\Fractal[\Phi^{-1}]$ of $\Phi^{-1}$ is the unique fixed point of the contracting function $\Phi^{-1}:\Comp(X)\to\Comp(X)$. Since $\Phi^{-1}(\mu)=\mu$, we conclude that $\Fractal[\Phi^{-1}]$ coincides with the Cantor set $\mu$.
\end{proof}

Next, we prove that the fixed fractal $\Fractal[\Phi]$ of the  multi-function $\Phi:x\mapsto\{3x,3x-2\}$ is the union of two isometric copies of the macro-Cantor set
$$M=\Big\{\sum_{k=0}^{\infty} 2x_k3^{k}:(x_k)_{k\in\w}\in\{0,1\}^\w,\;\sum_{k=0}^\infty x_i<\infty\Big\},$$
well-known in the Asymptotic Topology as the asymptotic counterpart of the Cantor set, see \cite{BZ}, \cite{DZ}.

\begin{proposition} The fixed fractal $\Fractal[\Phi]$ coincides with the union $(-M)\cup(M+1)$.
\end{proposition}

\begin{proof} Observe that $\Fixx[\Phi]=\{0,1\}\subset\IZ$ and $\Phi(\IZ)\subset\IZ$. Consequently, the fixed fractal $\Fractal[\Phi]=\Phi^\w(\{0,1\})\subset\IZ$ is discrete and by Proposition~\ref{p2:rep}, the multi-function $\Phi^{-1}$ is $*$-repelling.

By induction, we shall prove that for every $n\in\w$
$$\Phi^n(0)=\big\{-\sum_{i=0}^{n-1}2x_i3^i:(x_i)_{i=0}^{n-1}\in\{0,1\}^n\big\}.$$
This equality holds for $n=0$. Assume that for some $n\in\w$ this equality has been proved. Then
$$
\begin{aligned}
\Phi^{n+1}(0)&=3\cdot\Phi^n(0)\cup(3\cdot\Phi^n(0)-2)=\\
&=\Big\{-\sum_{i=0}^{n-1}2x_i3^{i+1},-2-\sum_{i=0}^{n-1}2x_i3^{i+1}:(x_i)_{i=0}^{n-1}\in\{0,1\}^n\Big\}=\\
&=\Big\{-\sum_{i=1}^{n}2x_i3^{i},-2-\sum_{i=1}^{n}2x_i3^{i}:(x_i)_{i=1}^{n}\in\{0,1\}^n\Big\}=\\
&=\Big\{-\sum_{i=0}^{n}2x_i3^{i}:(x_i)_{i=0}^{n}\in\{0,1\}^{n+1}\Big\}.
\end{aligned}
$$
Consequently, the orbit $\Phi^\w(0)$ of zero coincides with the set $-M$. By analogy we can prove that $\Phi^\w(1)=1+M$. So, $\Fractal[\Phi]=\Phi^\w(\{0,1\})=\Phi^\w(0)\cup\Phi^\w(1)=(-M)\cup(M+1)$.
\end{proof}

\begin{remark} The coarse characterization of the macro-Cantor set $M$ given in \cite{BZ} implies that the macro-fractal $\Fractal[\Phi]$ of the multi-function $\Phi:x\mapsto\{3x,3x-2\}$ is coarsely equivalent to $M$, so it is legal to call the macro-fractal $\Fractal[\Phi]$ a macro-Cantor set.
\end{remark}

\section{Algorithms of drawing fixed fractals}\label{s:algo}

In this section we describe some algorithms for drawing fixed fractals of contracting or $*$-repelling multi-functions. In fact, for contracting multi-functions such algorithms are well-developed and we refer the reader to \cite{Barnsley}, \cite{Crownover} for details. So, here we concentrate at the problem of drawing fixed fractals of $*$-repelling multi-functions.

From now on we assume that $\Phi:X\setmap X$ is a $*$-repelling finite-valued function on a complete metric space $(X,d)$ whose inverse $\Phi^{-1}$ is finite-valued and contracting with Lipschitz constant $\lambda=\Lip(\Phi^{-1})<1$. Then Theorems~\ref{t:micro} and Proposition~\ref{p2:rep} imply that:
\begin{enumerate}
\item the set $\Fixx[\Phi]=\Fixx[\Phi^{-1}]$ is finite and not empty;
\item the micro-fractal $\Fractal[\Phi^{-1}]$ is compact and not empty;
\item the macro-fractal $\Fractal[\Phi]$ is discrete in $X$ and coincides with the orbit $\Phi^\w(\Fixx[\Phi])$ of the finite set $\Fixx[\Phi]$;
\item there is $k\in\IN$ such that the finite set $\Phi^{k}(\Fixx[\Phi])\setminus\Phi^{k-1}(\Fixx[\Phi])$ lies on the positive distance
    $$\delta=d(\Phi^{k+1}(\Fixx[\Phi])\setminus\Phi^k(\Fixx[\Phi]),\Fractal[\Phi^{-1}])>0$$from the compact set $\Fractal[\Phi^{-1}]$.
\end{enumerate}

Now we describe some algorithms for drawing the macro-fractal $\Fractal[\Phi]$ and show the corresponding pictures for simple expanding maps $\Phi$.

\subsection{Cut-and-Zoom Algorithm} The {\em Cut-and-Zoom Algorithm}, draws the intersection $\Fractal[\Phi]\cap E$ of the macro-fractal $\Fractal[\Phi]$ with a ``large'' compact subset $E\subset X$ (which models a computer screen).

Given a compact subset $E\subset X$, we find $n\in\w$ such that
$$E\subset \{x\in X:d(x,\Fractal[\Phi^{-1}])<\delta/\lambda^n\}$$and draw the set $A=E\cap\Phi^{k+n}(\Fixx[\Phi])$. By Theorem~\ref{t:macro}, the set $A$ coincides with the intersection $E\cap\Fractal[\Phi]$ of the fractal $\Fractal[\Phi]$ with the ``screen'' $E$.

For drawing the set $\Phi^{k+n}(\Fixx[\Phi])$ we can use the following recurrent procedure.

\begin{lemma}\label{l:a} For every $m\in\w$ the set  $\Phi^{m}(\Fixx[\Phi])$ coincides with the union $\bigcup_{i=0}^{m}A_i$ of the sequence $(A_i)_{i=0}^{m}$ defined recursively as $A_0=\Fixx[\Phi]$ and $A_{i+1}=\Phi(A_i)\setminus A_i$ for $i<m$.
\end{lemma}

\begin{proof} The lemma is trivially true for $m=0$ and $m=1$. Assume that the equality
$\Phi^m(\Fixx[\Phi])=\bigcup_{i=0}^m A_i$ is true for some $m\ge 1$.
Then $A_{m+1}=\Phi(A_m)\setminus A_m\subset\Phi(\Phi^m(\Fixx[\Phi]))=\Phi^{m+1}(\Fixx[\Phi])$. To prove the reverse inclusion, fix any point  $x\in\Phi^{m+1}(\Fixx[\Phi])$. If $x\in \Phi^m(\Fixx[\Phi])$, then $$x\in \Phi^m(\Fixx[\Phi])=\bigcup_{i=0}^mA_i\subset\bigcup_{i=0}^{m+1}A_i$$by the inductive assumption.
Now assume that $x\in \Phi^{m+1}(\Fixx[\Phi])\setminus\Phi^{m}(\Fixx[\Phi])$ and choose a point $z\in\Phi^m(\Fixx[\Phi])$ with $x\in\Phi(z)$. It follows from $x\notin\Phi^m(\Fixx[\Phi])$ that $z\notin\Phi^{m-1}(\Fixx[\Phi])$ and $x\notin A_m$. By the inductive assumption, $$z\in \Phi^m(\Fixx[\Phi])\setminus\Phi^{m-1}(\Fixx[\Phi])= \Big(\bigcup_{i=0}^{m}A_i\Big)\setminus\Big(\bigcup_{i=0}^{m-1}A_i\Big)\subset A_m.$$ Consequently, $x\in \Phi(z)\setminus A_m\subset\Phi(A_m)\setminus A_m=A_{m+1}$ and hence $x\in\bigcup_{i=0}^{m+1}A_i$.
\end{proof}

\begin{remark} In fact, even with help of the recursive procedure suggested by Lemma~\ref{l:a}, drawing the set $\Phi^n(\Fixx[\Phi])$ for large $n$ is not an easy task. For a typical finite-valued function $\Phi$ the cardinality of the set $\Phi^n(\Fixx[\Phi])$ growth exponentially with growth of $n$. As a result, for relatively large $n$ the set $\Phi^n(\Fixx[\Phi])$ cannot be directly drawn on computer because of technical restrictions (the bounded amount of computer memory). This problem can be (partially) resolved by application of probability algorithm for drawing the set $\Phi^n(\Fixx[\Phi])$.

According to this algorithm one draws a huge number of random sequences $(x_i)_{i=0}^n$ where $x_0$ is a random point of the set $\Fixx[\Phi]$ and for every $i<n$, $x_{i+1}$ is a random point of the set $\Phi(x_i)$. Then with high probability (which can be effectively estimated) the union of such sequences will coincide with the finite set $\Phi^n(\Fixx[\Phi])$. With even higher probability this union will (visually) approximate the set $\Phi^n(\Fixx[\Phi])$.

In fact, all images of macro-fractals presented in this paper (except for Cut-and-Zoom images) were generated by a combined algorithm, which draws (with help of random algorithm) the orbits $\Phi^m(x)$ of points $x\in \Phi^n(\Fixx[\Phi])$ for relatively small $m$ and $n$. A program for drawing certain macro-fractals (with variable parameters) is available at http://testimages.somee.com
and we encourage the reader to experiment and produce his/her own macro-fractal images using this application.
\end{remark}

\subsection{Coloring macro-fractals} Observe that for a multi-valued function $\Phi:X\setmap X$ with finite set $\Fixx[\Phi]=\{x\in X:x\in\Phi(x)\}$ of fixed points, the fixed fractal $\Fractal[\Phi]=\bar\Phi^\w(\Fixx[\Phi])$ can be decomposed into a finite union
$$\Fractal[\Phi]=\bigcup_{x\in\Fixx[\Phi]}\bar\Phi^\w(x)$$
of closures of the orbit of individual fixed points $x\in\Fixx[\Phi]$. To see this structure of the fixed fractal in the following pictures of macro-fractals we shall color the orbits of fixed points by different colors. For micro-fractals such coloring is inexpedient as the orbit of each fixed point is dense in the micro-fractal.

Also, for some macro-fractals their black-and-white pictures show an interesting interference effects which disappear on their color counterparts; see Figures~\ref{macro-carpet-bw} and \ref{macro-carpet} or \ref{macro-sn7-bw} and \ref{macro-sn7}.

\begin{example}[The dual pair of fractal crosses] Consider the multi-function
$\Phi:\IC\to\IC$ assigning to each point $z\in\IC$ of the complex plane the 5-element subset $\Phi(z)=3z-2F$ where $F=\{0,1,-1,i,-1\}\subset\IC$. It is clear that the inverse multi-function $\Phi^{-1}:\IC\setmap\IC$, $\Phi^{-1}:w\mapsto \frac13w+\frac23F$, is contracting with Lipschitz  constant $\lambda=\Lip(\Phi^{-1})=\frac13<1$ and $\Fixx[\Phi]=\Fixx[\Phi^{-1}]=F$.

So, the fixed fractal $\Fractal[\Phi^{-1}]$ of the contracting multi-function $\Phi^{-1}$ is a micro-fractal called the {\em Micro-Fractal Cross}, drawn in Figure~\ref{micro-cross}.

\begin{figure}[h!]
  \begin{center}
    \includegraphics[width=0.5\linewidth, keepaspectratio]{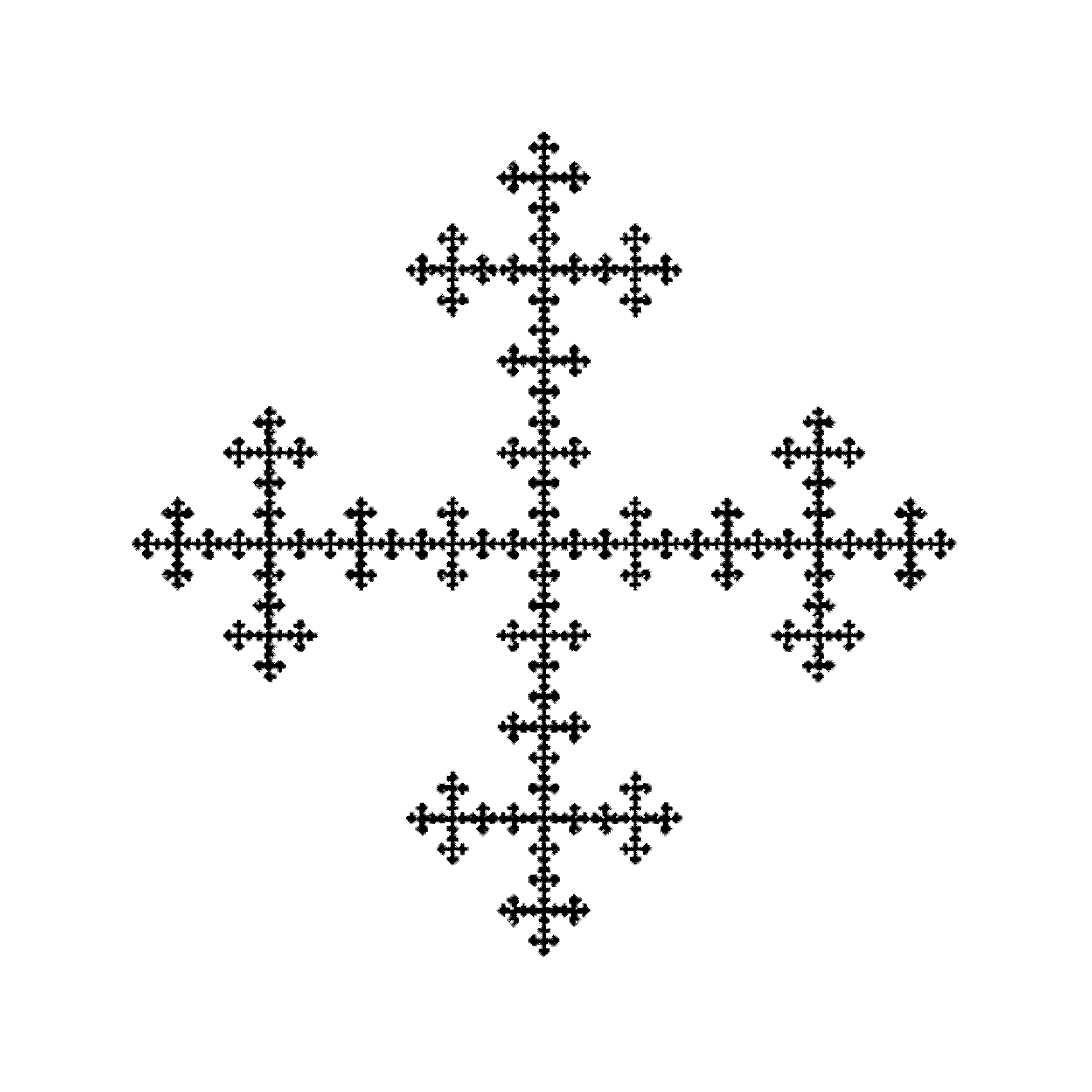}
   \end{center}
   \caption{Micro-Fractal Cross.}\label{micro-cross}
\end{figure}

\medskip

It follows from the inclusions $F\subset\IZ+i\hskip1pt\IZ$ and $\Phi(\IZ+i\hskip1pt\IZ)\subset\IZ+i\hskip1pt\IZ$ that the fixed fractal $\Fractal[\Phi]$ is discrete (being a subspace of the discrete subspace $\IZ+i\hskip1pt\IZ$ of the complex plane).
By Proposition~\ref{p2:rep}, the multi-function $\Phi$ is $*$-repelling. Consequently, its fixed fractal $\Fractal[\Phi]$ is a macro-fractal, which will be called the {\em Macro-Fractal Cross}. Observe that the 20-element set $\Phi(\Fixx[\Phi])\setminus\Fixx[\Phi]$ lies on the distance
$$\delta=d(\Phi(\Fixx[\Phi])\setminus\Fixx[\Phi],\Fractal[\Phi^{-1}])=1$$
from the Micro-Fractal Cross $\Fractal[\Phi^{-1}]$.
\pagebreak

Now for every $n\in\IN$ consider the neighborhood
$$E_n=\{x\in X:d(x,\Fractal[\Phi^{-1}])<3^n\}$$ of the micro-fractal $\Fractal[\Phi^{-1}]$. By Theorem~\ref{t:macro}, the intersection $E_n\cap\Fractal[\Phi]$ coincides with the set $E_n\cap\Phi^{n+1}(\Fixx[\Phi])$, which contains $\le 5^{n+1}=|\Phi^{n+1}(\Fixx[\Phi])|$ points.
\smallskip

Figure~\ref{macro-cross} shows the piece $E_n\cap\Fractal[\Phi]$ of the Macro-fractal Cross $\Fractal[\Phi]$ for $n=9$.

\begin{figure}[h!]
  \begin{center}
    \includegraphics[width=0.9\linewidth, keepaspectratio]{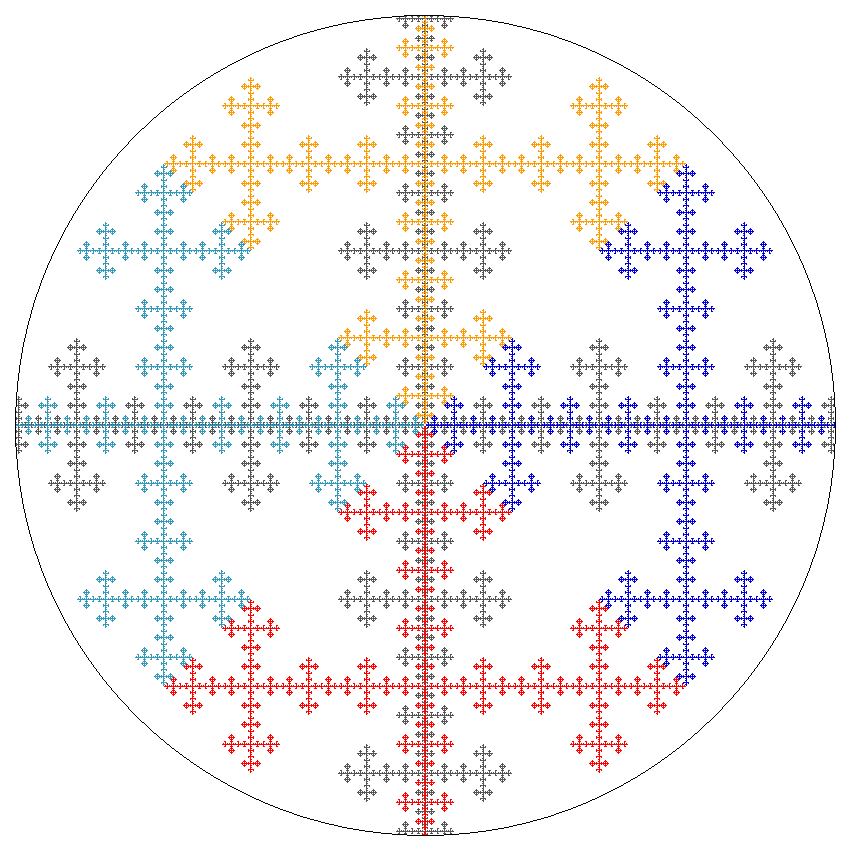}
   \end{center}
   \caption{The Macro-fractal Cross. Cut-and-Zoom image.}\label{macro-cross}
 \end{figure}
\end{example}
\pagebreak

\subsection{Non-linear images of macro-fractals} The problem with the Cut-and-Zoom Algorithm is that it shows only a part of the macro-fractal. To see the whole fractal at once, we can apply a suitable transformation to the ambient space $X$  in order to fit it into the computer screen.

To formalize this approach, let us fix some continuous map $i:X\to E$ of the complete metric space $X$ to a compact metric space $(E,\rho)$ called the {\em screen}. On the screen $E$ we would like to see the image $i(\Fractal[\Phi])$ of the macro-fractal $\Fractal[\Phi]$. The metric $\rho$ of the screen $E$ induces the Hausdorff distance $\rho_H$ on the power-set $2^E$ of all subsets of $E$.

\begin{proposition}\label{l:aa} For every $\e>0$ there is $N\in\IN$ such that
$$\rho_H\big(i(\Fractal[\Phi]),i(\Phi^n(\Fixx[\Phi])\big)<\e$$for all $n\ge N$.
\end{proposition}

\begin{proof} By the continuity of the map $i:X\to E$, the set $i\big(\Phi^\w(\Fixx[\Phi])\big)$  is dense in the image $i(\bar\Phi^\w(\Fixx[\Phi]))=i(\Fractal[\Phi])$ of the macro-fractal $\Fractal[\Phi]$. Then using the compactness of the metric space $(E,\rho)$ we can find a finite subset $F\subset i\big(\Phi^\w(\Fixx[\Phi])\big)$ which is an $\e$-net for $i(\Fractal[\Phi])$, in the sense that for each point $x\in i\big(\Fractal[\Phi]\big)$ there is a point $y\in F$ with $\rho(x,y)<\e$. Then $\rho_H(F,i(\Fractal[\Phi])<\e$.

The finite set $F$ lies in the image $i(\Phi^N(\Fixx[\Phi]))$ for some $N\in\w$. Then for every $n\ge N$, the inclusions
$$F\subset i(\Phi^n(\Fixx[\Phi]))\subset i(\Fractal[\Phi])$$imply that $$\rho_H\big(i(\Phi^n(\Fixx[\Phi])),i(\Fractal[\Phi])\big)\le\rho_H(F,\Fractal[\Phi])<\e.$$
\end{proof}

According to Proposition~\ref{l:aa}, for a large $n$ the image $i(\Phi^n(\Fixx[\Phi]))$ of the set $\Phi^n(\Fixx[\Phi])$ approximates the image $i(\Fractal[\Phi])$ in the Hausdorff metric $\rho_H$, so it approximates the macro-fractal also in the visual sense.
In the next subsections we shall consider images of the macro-fractal $\Fractal[\Phi]$ on various screens $E$.

\subsection{Spherical images of macro-fractals}

Let us observe that Proposition~\ref{l:aa} says that the image $i(\Fractal[\Phi])$ can be approximated by the images  $i\big(\Phi^n(\Fixx[\Phi])\big)$ of the sets $\Phi^n(\Fixx[\Phi])$ for sufficiently large $n$ but does not say how large this $n$ should be taken to get an approximation with a given precision.

This can be easily done for the screen $E$ which coincides with the  one-point compactifications of $X$. In this case we should assume that $X$ is a separable locally compact space, so its one-point compactification $E=X\cup\{\infty\}$ is compact and metrizable.

We recall that the {\em one-point compactification} of a locally compact space $(X,\tau_X)$ is the union $\alpha X=X\cup\{\infty\}$ of $X$ and some point $\infty\notin X$, endowed with the topology $$\tau_{\alpha X}=\tau_X\cup\{\alpha X\setminus K:\mbox{$K$ is a compact subset of $X$}\}.$$

Since the space $E=\alpha X$ is metrizable, we can fix a metric $\rho$ generating the topology of the compact metrizable space $E$ and consider the identity embedding $i:X\to E$. This embedding $i:X\to E=\alpha X$ will be called the {\em spherical embedding} of $X$ into its {\em Riemann sphere} $E=\alpha X$.
\smallskip

\noindent {\bf Algorithm of spherical drawing macro-fractals.} To draw the spherical image $i(\Fractal[\Phi])$ of the macro-fractal $\Fractal[\Phi]$ with precision $\e>0$ we should make the following steps:
\begin{enumerate}
\item find a compact subset $K\subset X$ such that $X\setminus K$ lies in the $\e$-neighborhood $O_\e(\infty)$ of the compactifying point $\infty$;
\item find $m\in\w$ such that $K\subset \{x\in X:d(x,\Fractal[\Phi^{-1}])\le\delta/\lambda^m\}$;
\item draw the set $A=\{\infty\}\cup i(\Phi^{k+m}(\Fixx[\Phi]))$.
\end{enumerate}
This set $A$ will approximate the image $i(\Fractal[\Phi])$ with precision $\rho_E(A,i(\Fractal[\Phi]))<\e$. In this algorithm, $\lambda=\Lip(\Phi^{-1})<1$ and $\delta=d(\Phi^{k}(\Fixx[\Phi])\setminus\Phi^{k-1}(\Fixx[\Phi]),\Fractal[\Phi^{-1}])>0$.
\pagebreak

If $X$ is an Euclidean space $\IR^n$, then its one-point compactification $\alpha X$ can be identified with the $n$-dimensional sphere $S^{n}_r=\{\mathbf x\in\IR^{n+1}:\|\mathbf x\|=r\}$ of some radius $r$. The embedding of $X=\IR^n$ into the sphere $S^n_r$ is given by the inverse stereographic projection $s_\infty:\IR^n\to S^n_r$, which assigns to each point $\vec x=(x_0,x_1,\dots,x_{n-1})\in\IR^n$ the unique point $\vec y=(y_0,\dots,y_n)\in S^n$ that lies on the segment connecting the points $(0,\dots,0,r)$ and $(x_0,\dots,x_{n-1},0)$. The coordinates of the vector $\vec y$ can be calculated by the formula:
$$y_i=
\begin{cases}
\frac{2r^2}{r^2+\|\mathbf x\|^2}\cdot x_i&\mbox{if $i<n$}\\
\frac{\|\mathbf x\|^2-r^2}{\|\mathbf x\|^2+r^2}\cdot r&\mbox{if $i=n$}.
\end{cases}
$$
\smallskip

\begin{figure}[h!]
  \begin{center}
    \includegraphics[width=0.9\linewidth, keepaspectratio]{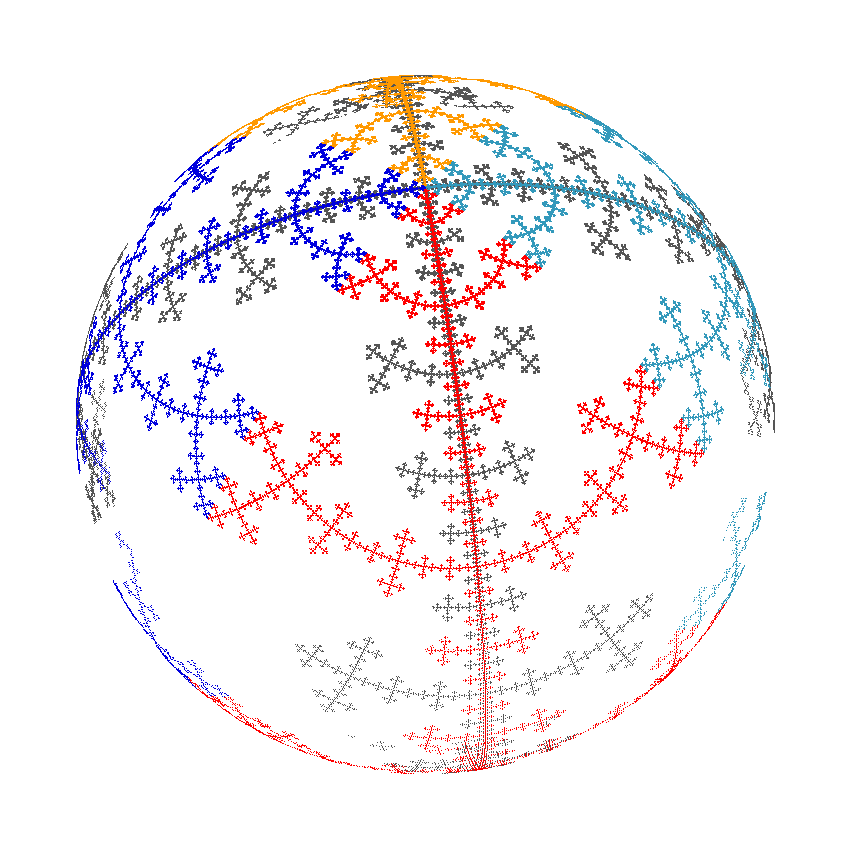}%
   \end{center}
   \caption{Macro-Fractal Cross. Spherical image.}\label{str-cross}
\end{figure}

\pagebreak

The spherical embedding allows us to look at the structure of the (image of the) macro-fractal at a neighborhood of infinity. Figure~\ref{sph-cross} shows the neighborhood of infinity on the spherical image of Macro-Fractal Cross.
\smallskip

\begin{figure}[h!]
  \begin{center}
    \includegraphics[width=0.7\linewidth, keepaspectratio]{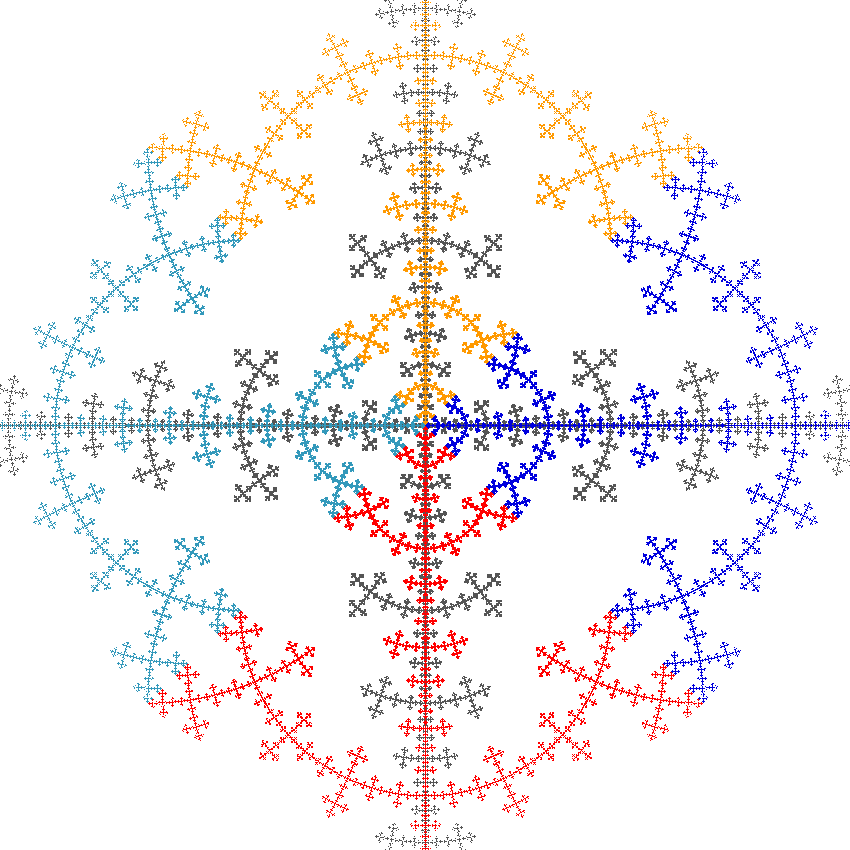}
   \end{center}
   \vskip15pt
   \caption{The spherical image of the Macro-Fractal Cross at a neighborhood of infinity.}\label{sph-cross}
  \end{figure}
\pagebreak

\subsection{Semi-spherical images of macro-fractals} In this section we consider another type on the transformation of the macro-fractals called the semi-spherical image. It is defined if $X=\IR^n$ is an Euclidean space. For this transformation, the screen $E$ is the lower semi-sphere of the sphere $S^n_r=\{\mathbf x\in\IR^{n+1}:\|\mathbf x\|=r\}$ of radius $r$ in the Euclidean space $\IR^{n+1}$. The semi-spherical reflection
$s_0:\IR^n\to S^n_r$ assigns to each point $\mathbf x=(x_0,\dots,x_{n-1})\in\IR^n$ the unique point $\mathbf y=(y_0,\dots,y_n)\in S^n_r$ on the segment connecting the center of the sphere with the point $(x_0,\dots,x_{n-1},-r)$. The coordinates of the vector $\mathbf y$ can be found by the formula:
$$
y_i=\begin{cases}\frac{rx_i}{\sqrt{r^2+\|\mathbf x\|^2}}&\mbox{if $i<n$}\\
-\frac{r^2}{\sqrt{r^2+\|\mathbf x\|^2}}&\mbox{if $i=n$}.
\end{cases}
$$

The image $s_0(\IR^n)$ of $\IR^n$ under the semi-spherical projection coincides with the lower semi-sphere $\{(y_0,\dots,y_n)\in S^n_r:y_n<0\}$. We can next project this semi-sphere on the disk $D_{2r}=\{\mathbf x\in\IR^n:\|x\|<2r\}$ using the stereographic projection $s_\infty^{-1}:S^n_r\setminus\{(0,\dots,0,r)\}\to\IR^n$.

The semi-spherical image of the Macro-Fractal Cross is drawn in Figures~\ref{semi-cross}.

\begin{figure}[h!]
  \begin{center}
    \includegraphics[width=0.90\linewidth, keepaspectratio]{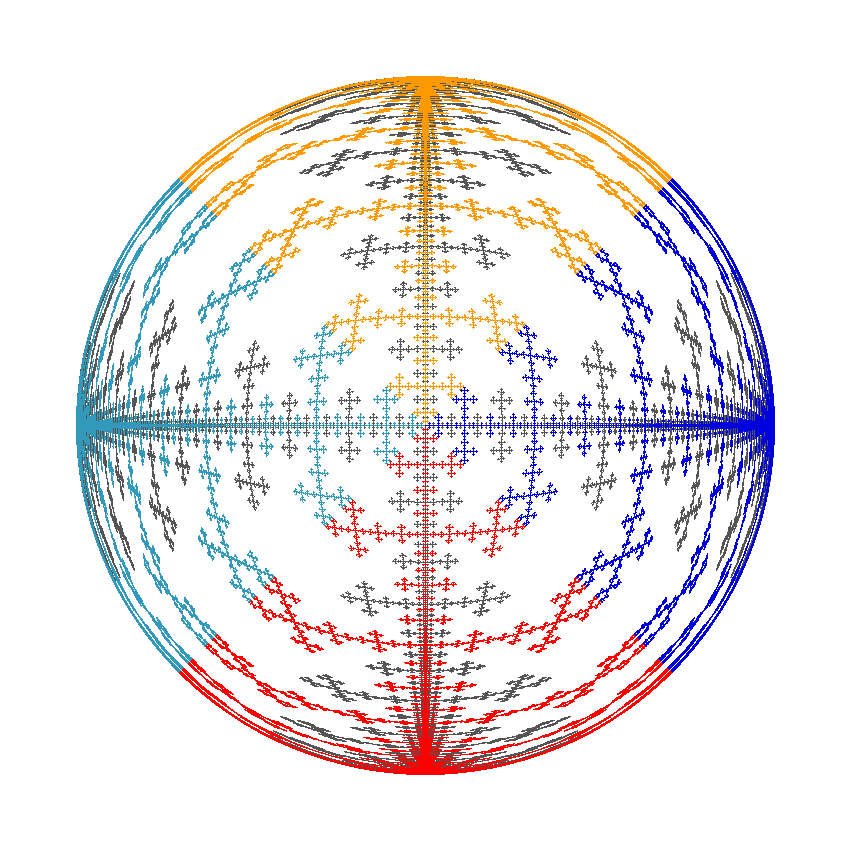}
   \end{center}
   \caption{The Macro-Fractal Cross. Semi-spherical image.}\label{semi-cross}
\end{figure}
\pagebreak

\subsection{Projective images of macro-fractals}

The projective image of a macro-fractal $\Fractal[\Phi]\subset X=\IR^n$ uses the projective space $\RP^n_r$ as a screen $E$. The projective space $\RP^n_r$ is the quotient space of the sphere $S^n_r$ of radius $r$ under the quotient map $q:S^n_r\to \RP^n_r$, which identifies opposite points of the sphere. The composition $\pi=q\circ s_0:\IR^n\to\RP^n_r$ is called the {\em projective reflection} of the space $X=\IR^n$. By Proposition~\ref{l:aa}, the image $\pi(\Fractal[\Phi])$ of the macro-fractal $\Fractal[\Phi]$ can be approximated by the images of the sets $\pi(\Phi^m(\Fixx[\Phi]))$ for large $m$.

In two-dimensional case the projective plane $\RP^2_r$ is locally diffeomorphic to the plane $\IR^2$. So, to see the behavior of the macro-fractal $\Fractal[\Phi]$ at infinity, we can apply a projective transformation to the plane such that the line of horizon transforms into a usual line on the plane and then the behavior of the macro-fractal at infinity becomes visible as the behavior of its projective image near the image of the horizon line.

The projective image of the Macro-fractal Cross is drawn in Figure~\ref{proj-cross}.
\vskip20pt

\begin{figure}[h!]
  \begin{center}
    \includegraphics[width=0.9\linewidth]{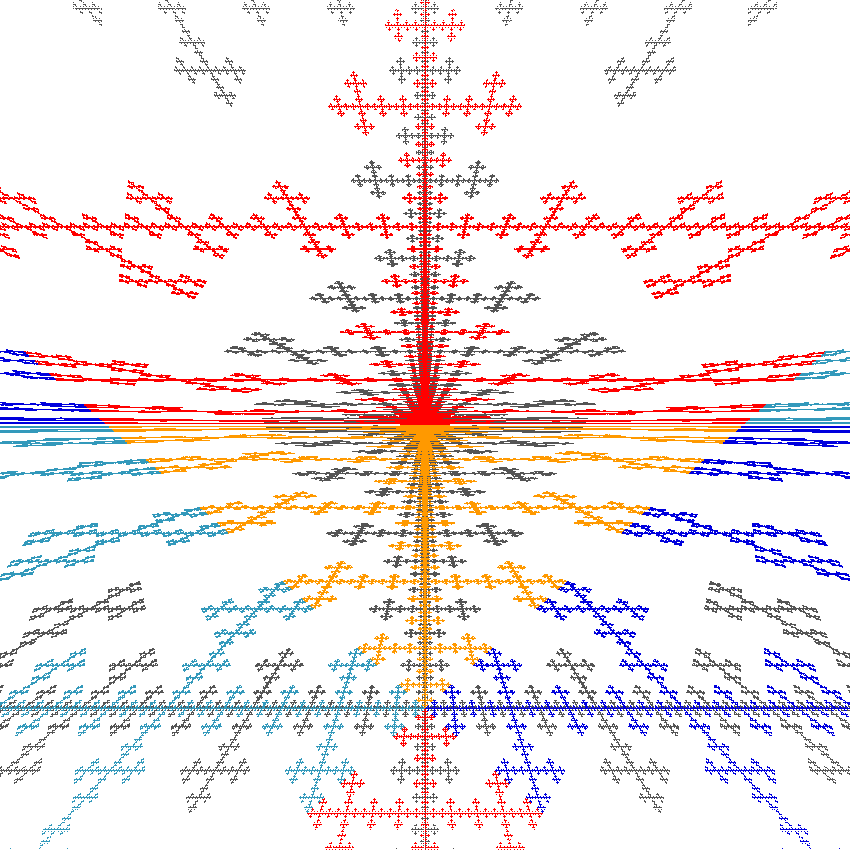}%
   \end{center}
   \vskip15pt
   \caption{The Macro-Fractal Cross. Projective image.}\label{proj-cross}
  \end{figure}

\pagebreak

\section{Acknowledgments}

The authors would like to express their sincere thanks to Ostap Chervak and Sasha Ravsky for many stimulating discussions on the theory of macro-fractals.

\newpage

\section{Appendix: Gallery of Micro and Macro Fractals}

\subsection{A dual fractal to the Sierpi\'nski triangle}
Consider an equilateral triangle with vertices $c_1,c_2,c_3$ in the complex plane $\IC$. Let $F=\{c_1,c_2,c_3\}$ and consider the multi-valued function
$$\Phi:\IC\to\IC,\;\;\Phi:z\mapsto 2z-F$$with inverse multi-function
$$\Phi^{-1}:\IC\setmap\IC,\;\;\Phi^{-1}:z\mapsto\frac12+\frac12F,$$
which is contracting with Lipschitz constant $\Lip(\Phi^{-1})=\frac12$.
So, $\Fractal[\Phi]$ is a macro-fractal, dual to the micro-fractal $\Fractal[\Phi^{-1}]$, called the {\em Sierpi\'nski triangle}.
A Sierpi\'nski equilateral triangle is drawn on Figure~\ref{micro-serp}. Various images of its dual macro-fractal are shown on Figures~\ref{macro-serp}--\ref{proj-serp}.

\begin{figure}[h!]
  \begin{center}
    \includegraphics[width=0.6\linewidth, keepaspectratio]{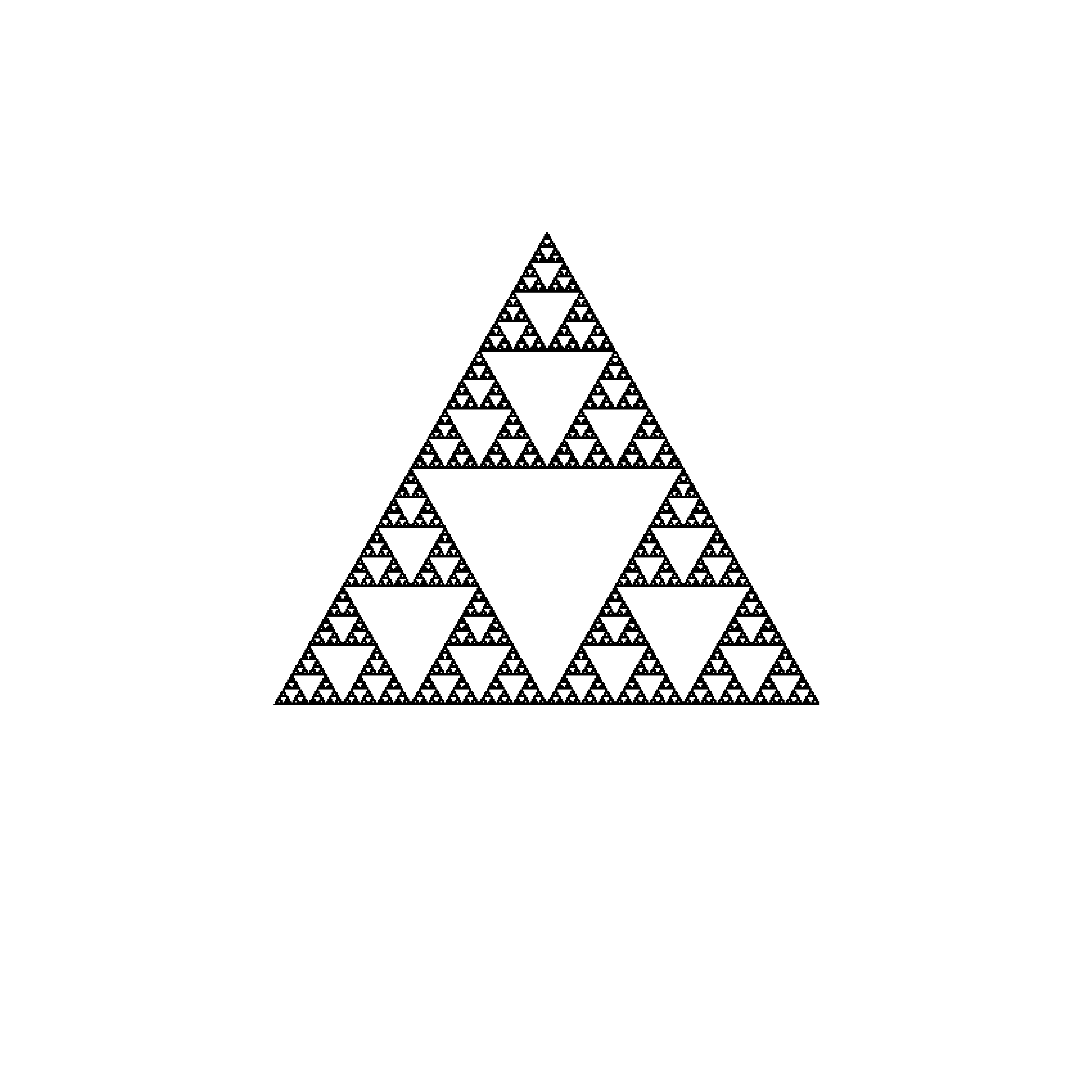}
      \end{center}{} \vskip-40pt
   \caption{Sierpi\'nski triangle.}\label{micro-serp}
\end{figure}

\phantom{mm}
\newpage

\hskip40pt
\phantom{mm}

\begin{figure}[h!]
  \begin{center}
    \includegraphics[width=0.9\linewidth, keepaspectratio]{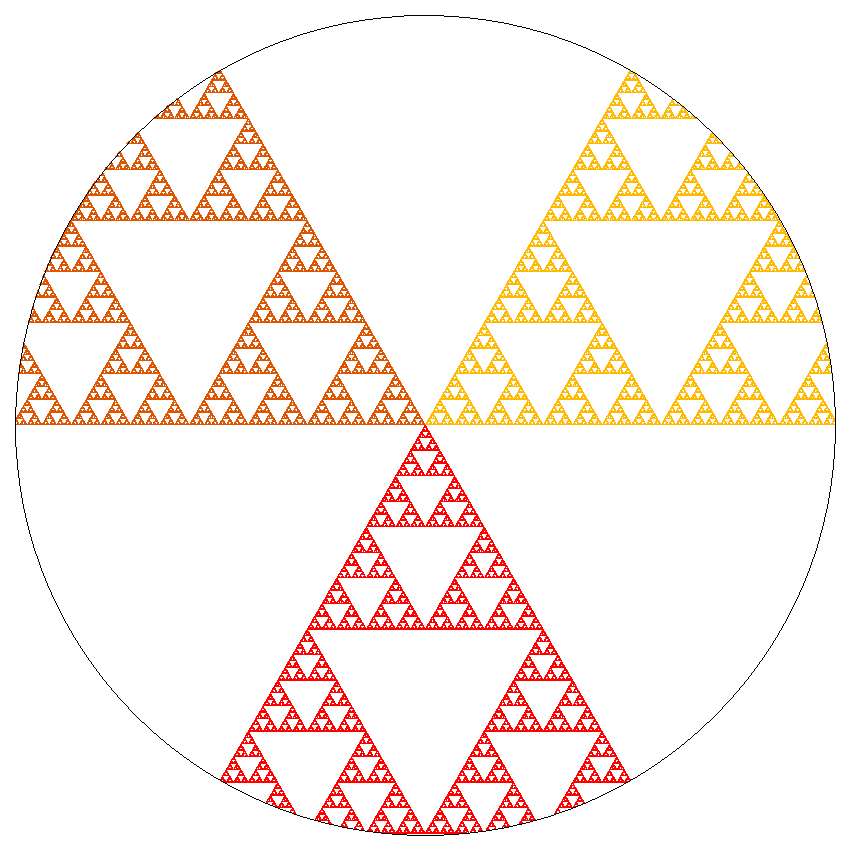}
   \end{center}  \vskip20pt
   \caption{A large piece of the dual fractal to the Sierpi\'nski triangle.}\label{macro-serp}
\end{figure}

\phantom{mm}
\pagebreak

\hskip40pt

\begin{figure}[h!]
  \begin{center}
    \includegraphics[width=1\linewidth, keepaspectratio]{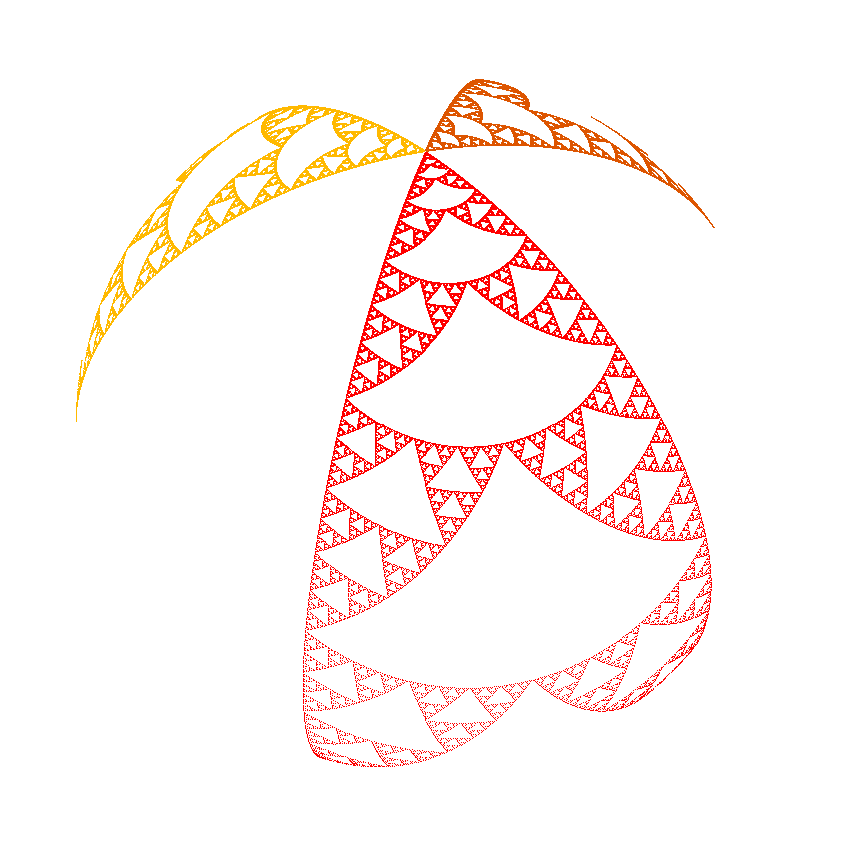}
   \end{center}
   \caption{The dual fractal to the Sierpi\'nski triangle.}{Spherical image.}\label{str-serp}
\end{figure}

\phantom{mm}
\pagebreak
\hskip40pt

\begin{figure}[h!]
  \begin{center}
    \includegraphics[width=0.9\linewidth, keepaspectratio]{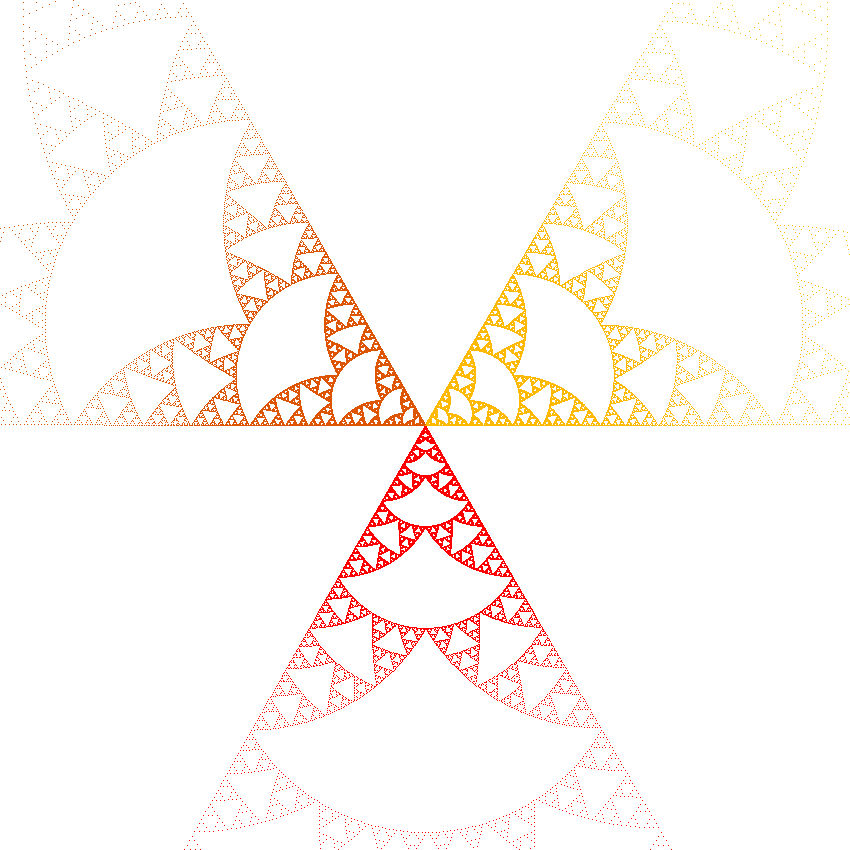}
  \end{center}  \vskip20pt
   \caption{The dual fractal to the Sierpi\'nski triangle.}{Spherical image at a neighborhood of infinity.}\label{sph-serp}
\end{figure}

\phantom{mm}
\pagebreak
\hskip40pt

\begin{figure}[h!]
  \begin{center}
    \includegraphics[width=1\linewidth]{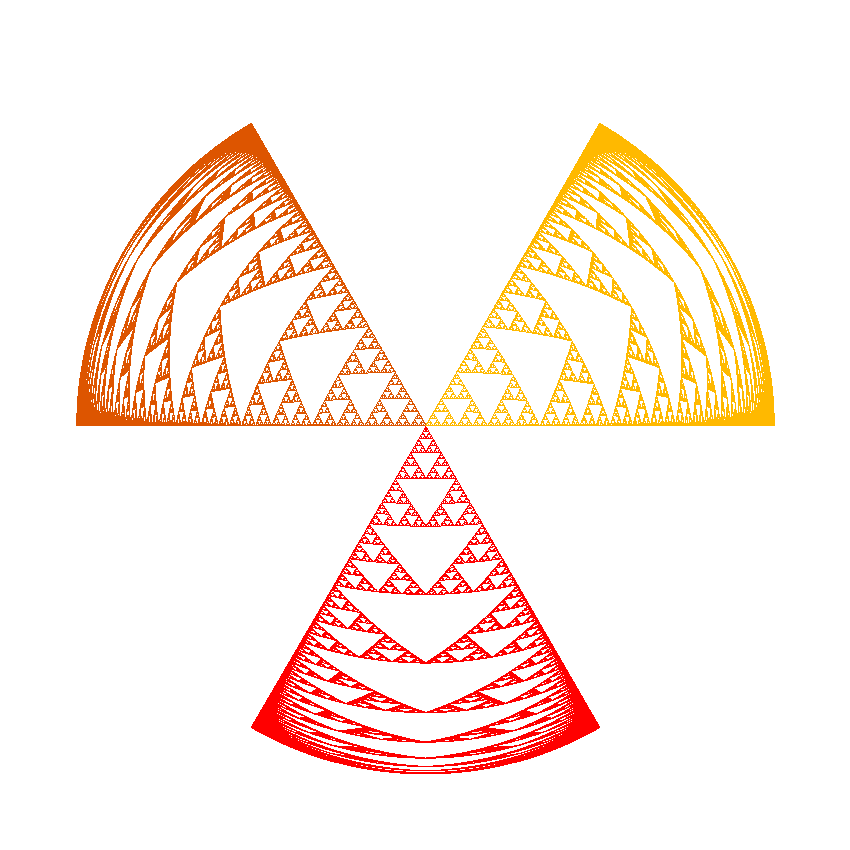}
   \end{center}
   \caption{The dual fractal to the Sierpi\'nski triangle.}{Semi-spherical image.}\label{semi-serp}
\end{figure}

\phantom{mm}
\pagebreak
\hskip40pt

\begin{figure}[h!]
  \begin{center}
    \includegraphics[width=0.9\linewidth]{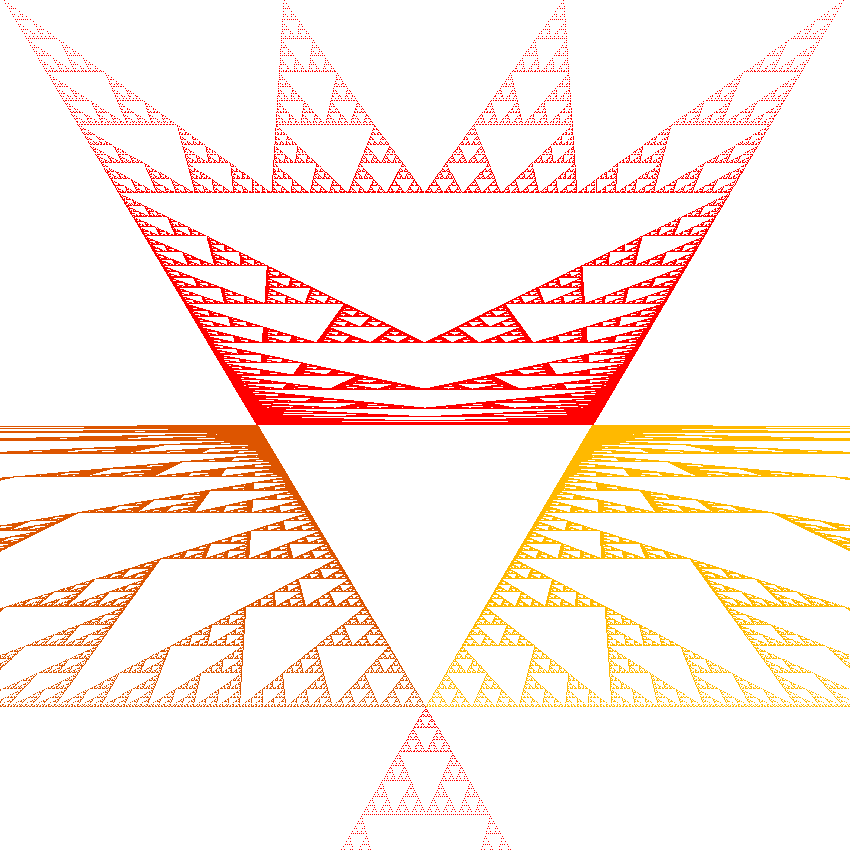}
      \end{center} \vskip20pt
   \caption{The dual fractal to the Sierpi\'nski triangle.}{Projective image.}\label{proj-serp}
\end{figure}\phantom{mm}
\pagebreak

\subsection{Dual Fractal to the Koch curve}
The Koch Curve is the micro-fractal $\Fractal[\Phi^{-1}]$ of the contracting multi-function $$\Phi^{-1}:\IC\setmap\IC,\;\;\Phi^{-1}:z\mapsto\big\{3+\tfrac1{\sqrt{3}}e^{-i\pi/3}(\bar z-3),-3+\tfrac1{\sqrt{3}}e^{i\pi/3}(\bar z+3)\big\},$$
which is inverse to the expanding multi-valued function
$$\Phi:\IC\setmap \IC,\;\;\Phi:z\mapsto\{3+\sqrt{3}e^{-i\pi/3}(\bar z-3),-3+\sqrt{3}e^{i\pi/3}(\bar z+3)\}$$
 where $\bar z=x-iy$ conjugated complex number to $z=x+iy$.
\bigskip

\begin{figure}[h]
  \begin{center}
    \includegraphics[width=0.9\linewidth, keepaspectratio]{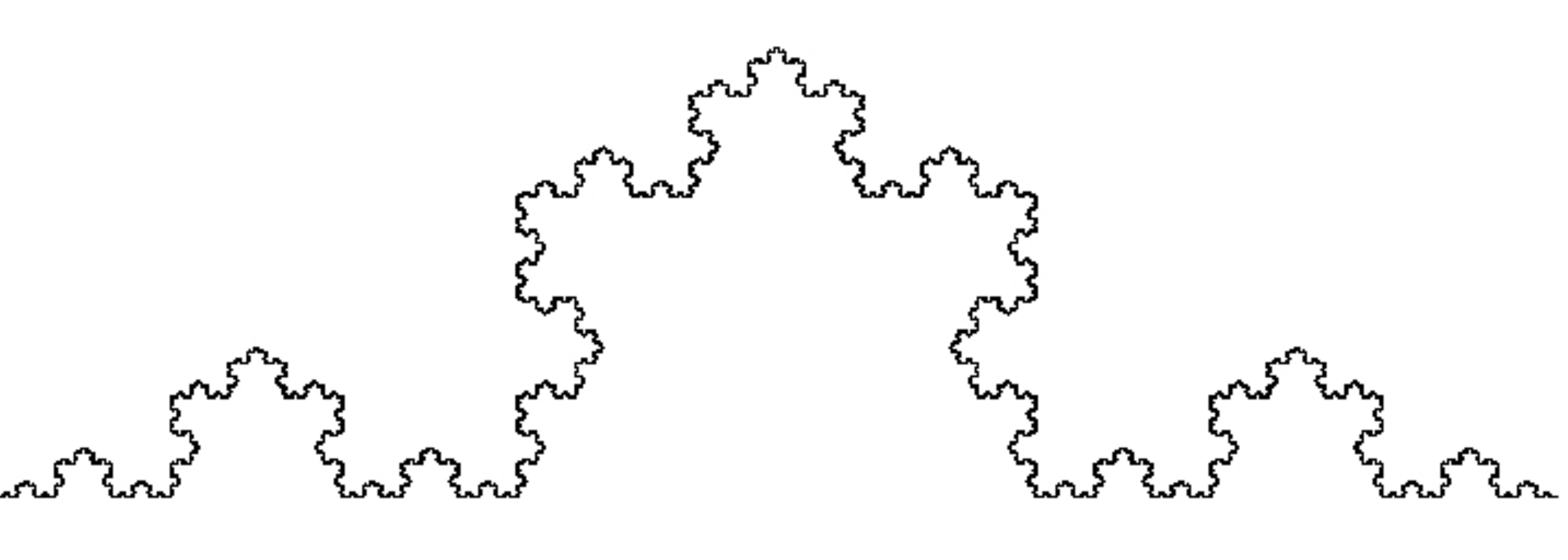}
   \end{center}
   \caption{Koch Curve.}\label{micro-koch}
\end{figure}
\vfill
\pagebreak
\phantom{m}
\vskip40pt

\begin{figure}[h!]
  \begin{center}
    \includegraphics[width=1\linewidth, keepaspectratio]{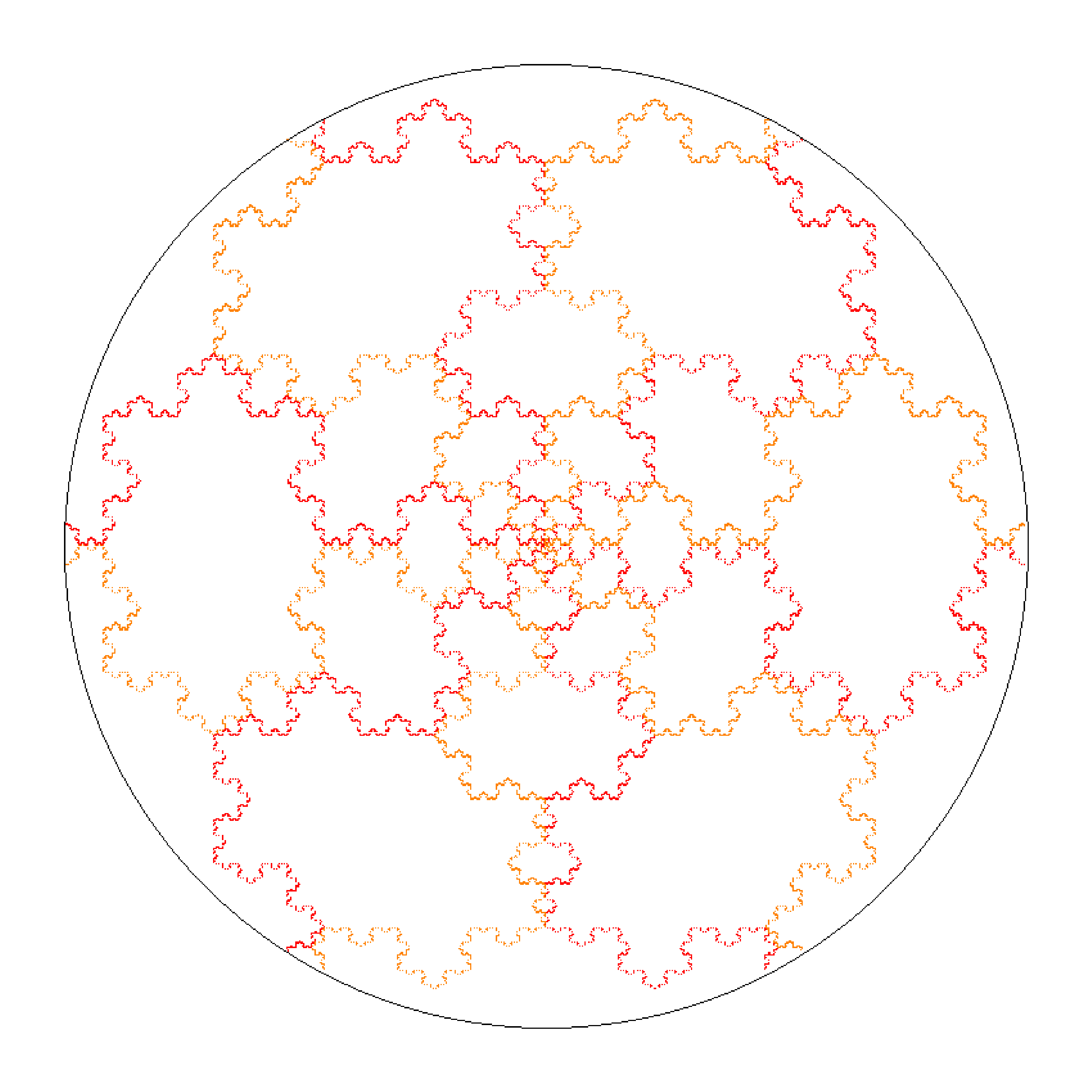}
   \end{center}
   \caption{The dual Fractal to the Koch Curve.}{Cut-and-Zoom image.}\label{macro-koch}
\end{figure}
\vfill

\pagebreak
\phantom{m}
\vskip40pt

\begin{figure}[h!]
  \begin{center}
    \includegraphics[width=1\linewidth,height=0.8\linewidth]{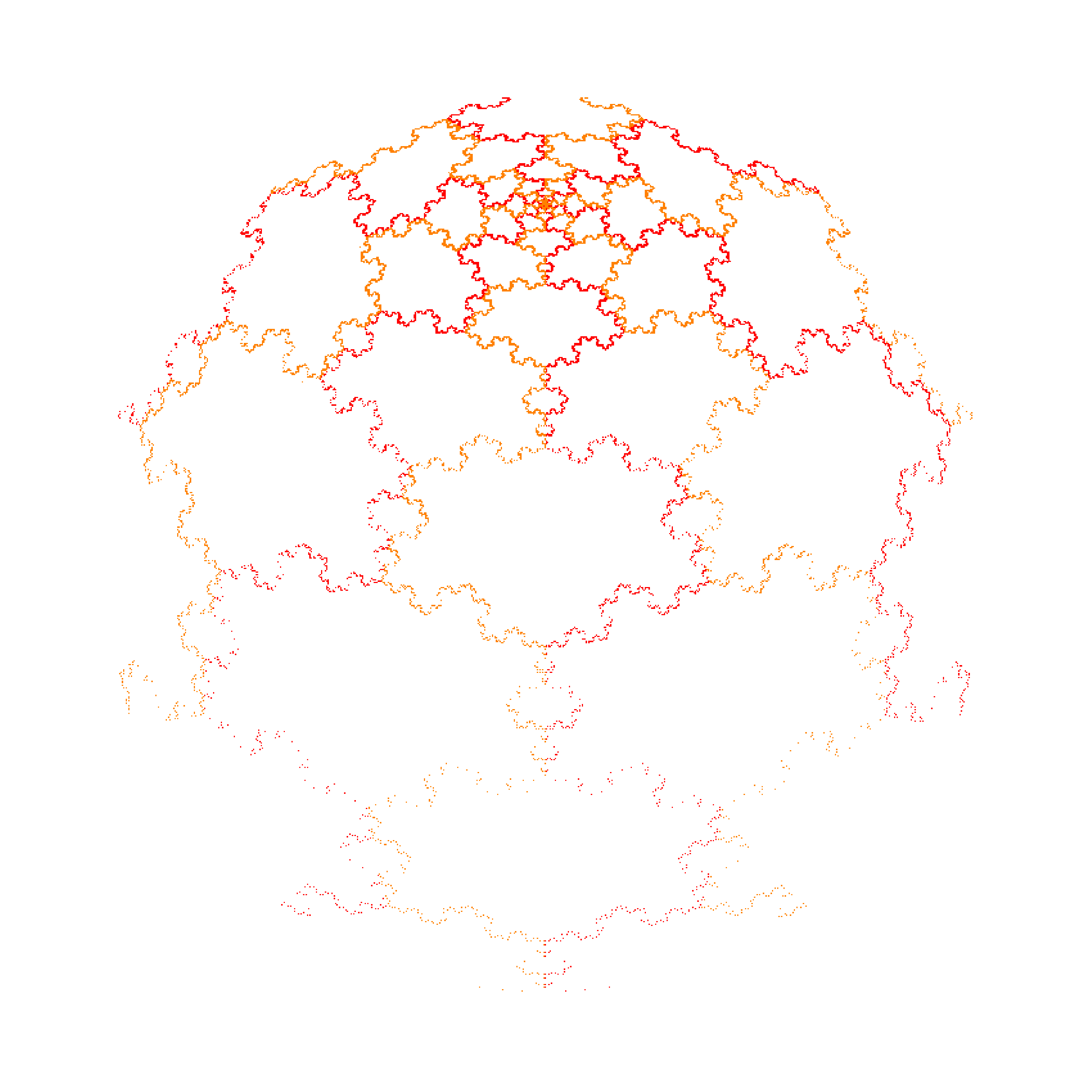}
      \end{center} \vskip20pt
   \caption{The dual fractal to the Koch Curve.}{Spherical image.}\label{str-koch}
\end{figure}
\vfill

\pagebreak
\phantom{m}
\vskip40pt

\begin{figure}[h!]
  \begin{center}
    \includegraphics[width=0.9\linewidth, keepaspectratio]{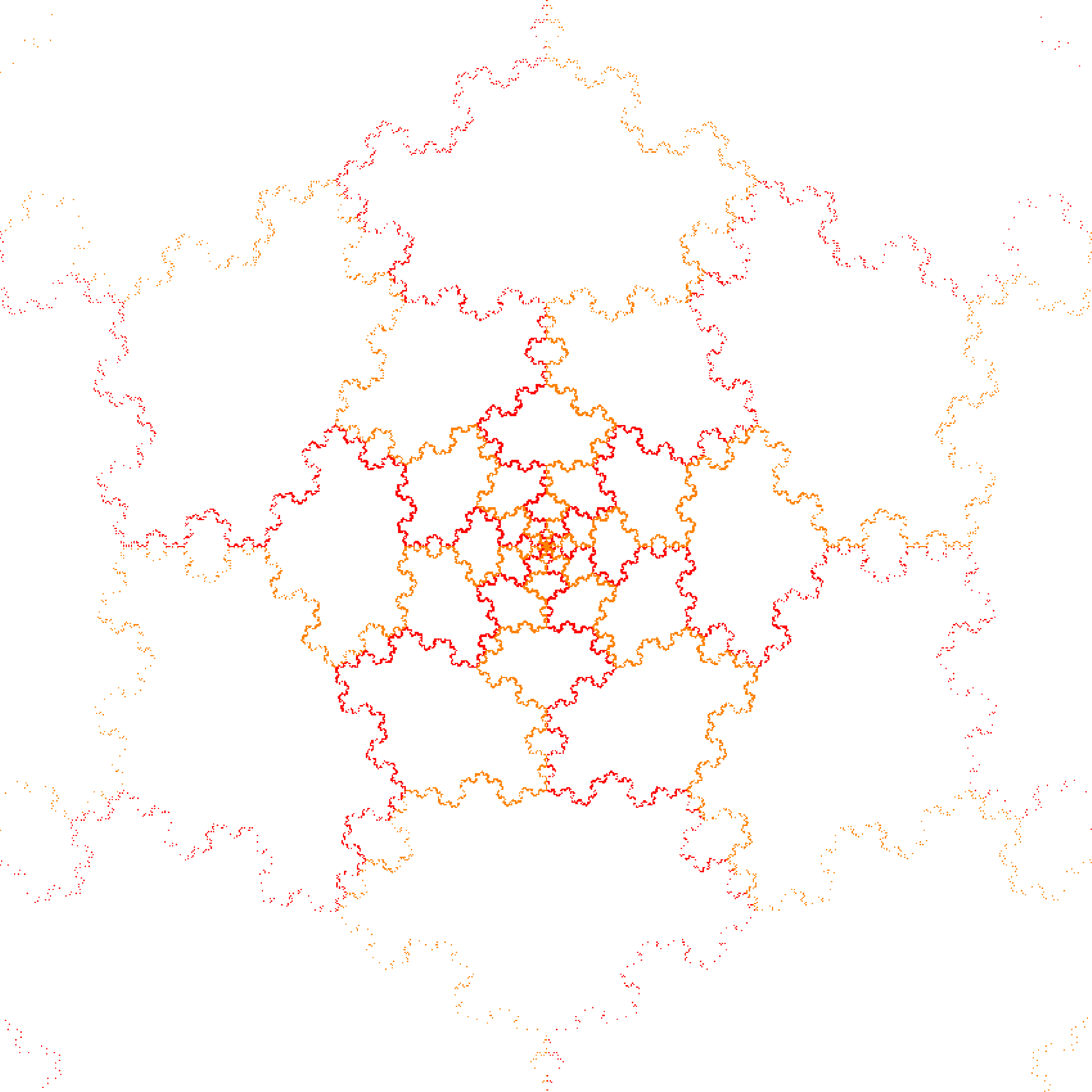}
       \end{center}\vskip20pt
   \caption{The dual fractal to the Koch curve.}{Spherical image at a neighborhood of the infinity.}\label{sph-koh}
\end{figure}
\vfill

\pagebreak
\phantom{m}
\vskip40pt

\begin{figure}[h!]
  \begin{center}
    \includegraphics[width=0.7\linewidth, keepaspectratio]{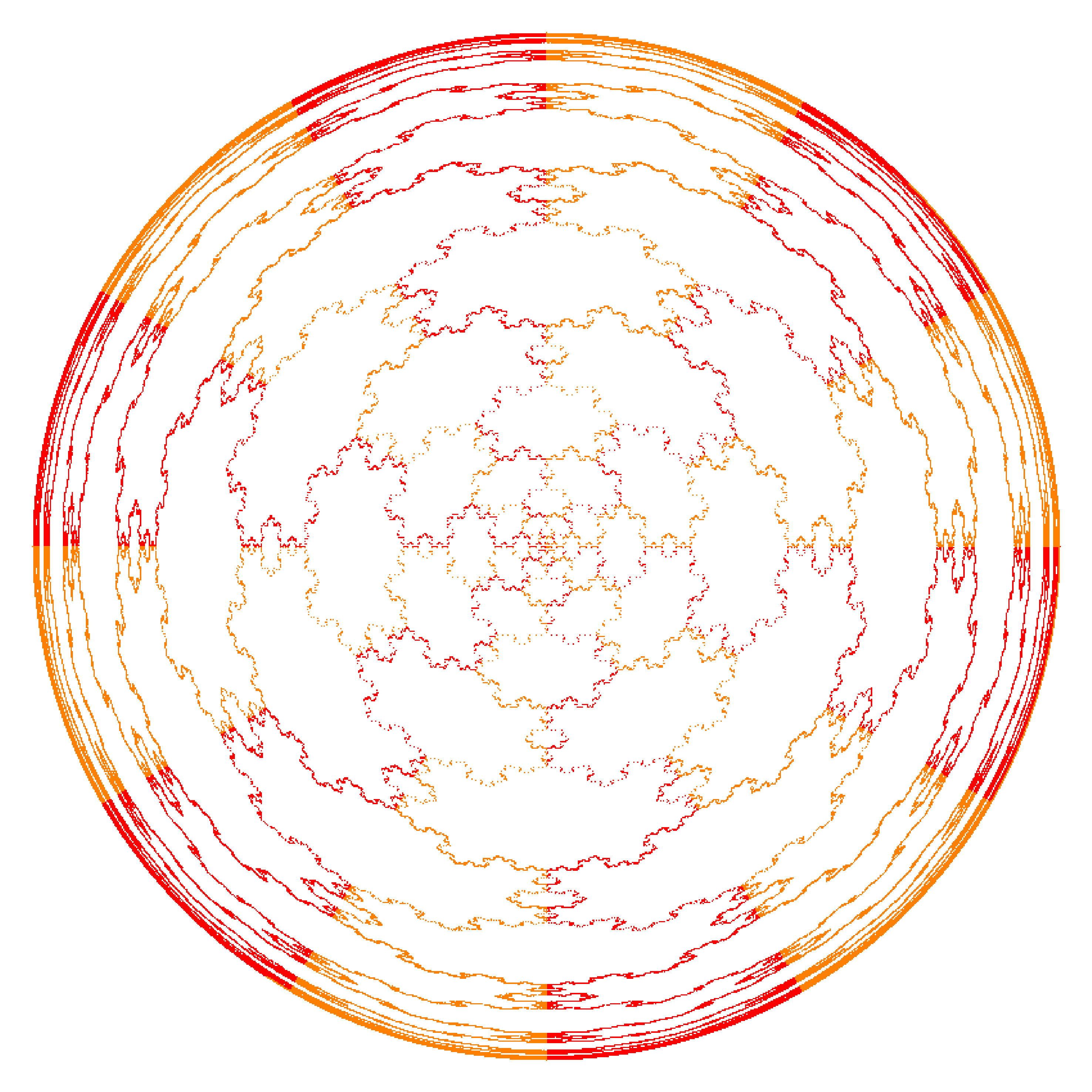}
     \end{center}\vskip30pt
   \caption{The dual fractal to the Koch Curve.}{Semi-spherical image.}\label{semi-koch}
\end{figure}
\vfill

\pagebreak
\phantom{m}
\vskip40pt

\begin{figure}[h!]
  \begin{center}
    \includegraphics[width=0.9\linewidth]{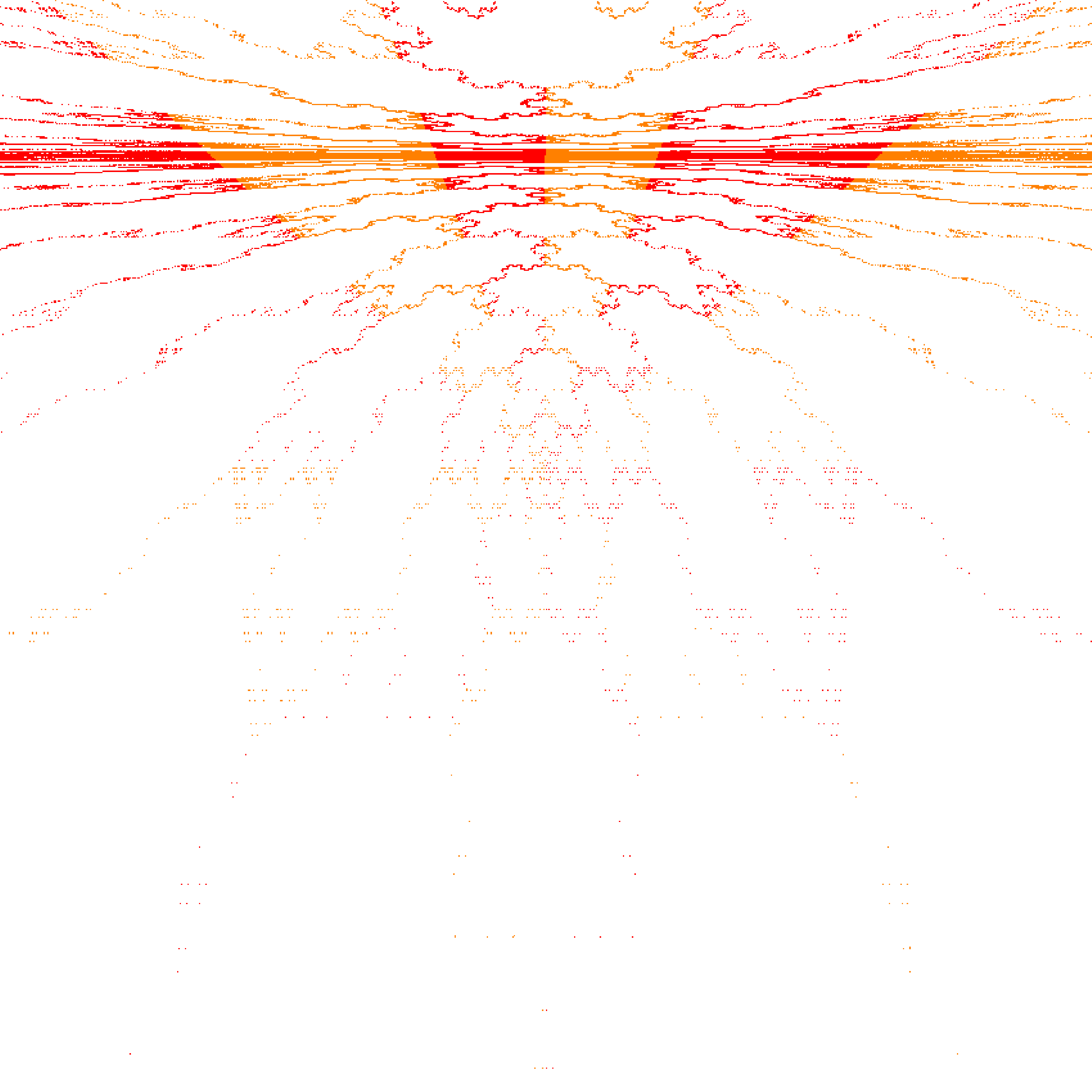}
   \end{center}
   \caption{The dual fractal to the Koch Curve.}{Projective image.}\label{proj-koch}
\end{figure}
\vfill

\pagebreak

\subsection{Micro- and Macro-Fractal Snowflakes with 6 fixed points}
Let $$F_6=\{e^{i\pi k/3}:0\le k<6\}$$be the 6-element set of vertices of a regular 6-gon on the complex plane. Consider the expanding multi-valued function
$$\Phi:\IC\setmap \IC,\;\;\Phi:z\mapsto 3z-2F_6$$whose inverse $$\Phi^{-1}:\IC\setmap\IC,\;\;\Phi^{-1}:z\mapsto \tfrac13z+\tfrac23F_6$$is a contracting multi-valued function. The dual fractals $\Fractal[\Phi]$ and $\Fractal[\Phi^{-1}]$ are called the {\em Macro-Fractal and Micro-Fractal Snowflakes with 6 fixed points}, respectively. Their images are drawn on Figures~\ref{micro-flake}--\ref{proj-sn6}.
\bigskip
\vskip30pt

\begin{figure}[h!]
  \begin{center}
    \includegraphics[width=0.6\linewidth, height=0.67\linewidth]{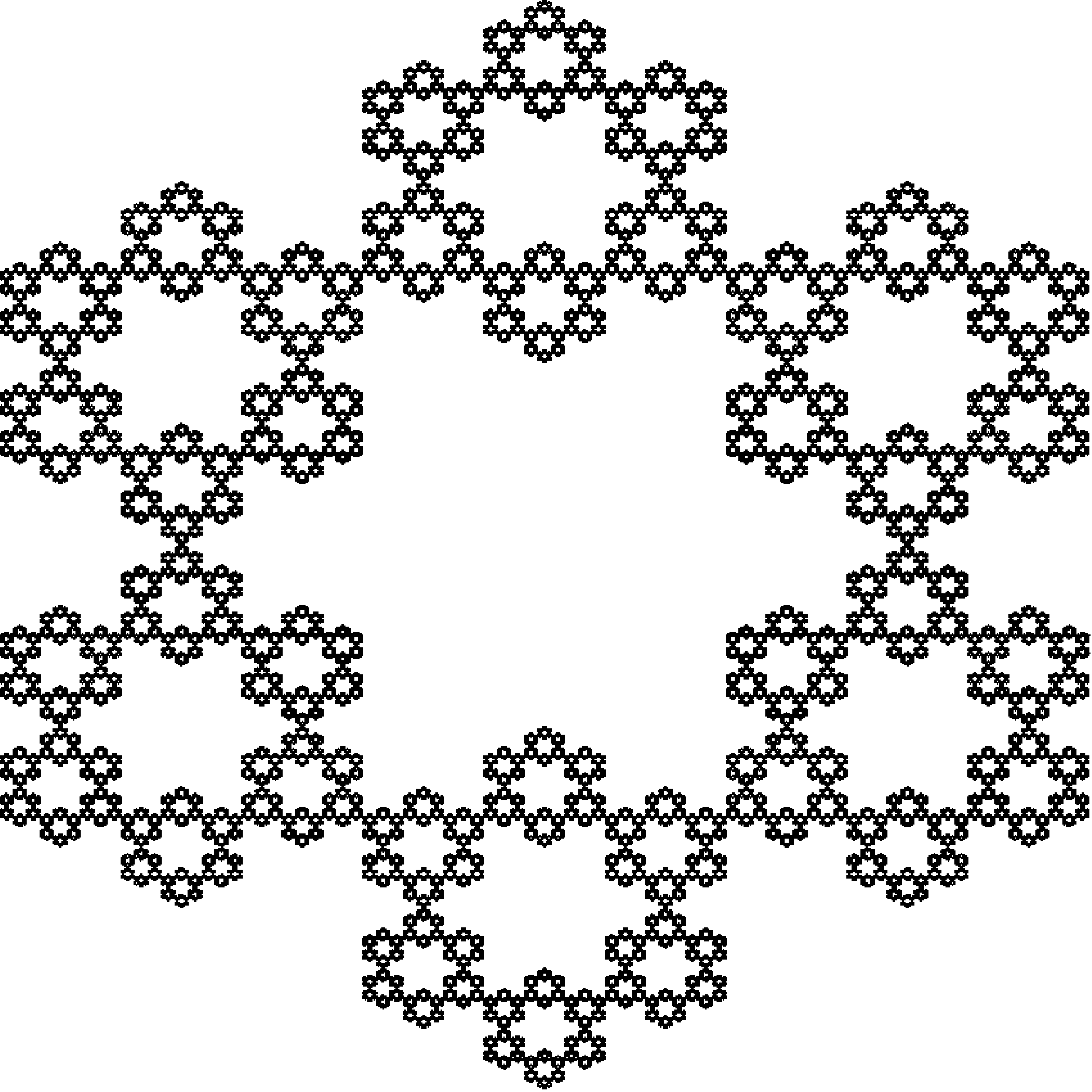}
    \end{center} \vskip30pt
   \caption{Micro-Fractal Snowflake with 6 fixed points.}\label{micro-flake}
\end{figure}
\vfill

\pagebreak
\phantom{m}
\vskip40pt

\begin{figure}[h!]
  \begin{center}
    \includegraphics[width=0.9\linewidth, keepaspectratio]{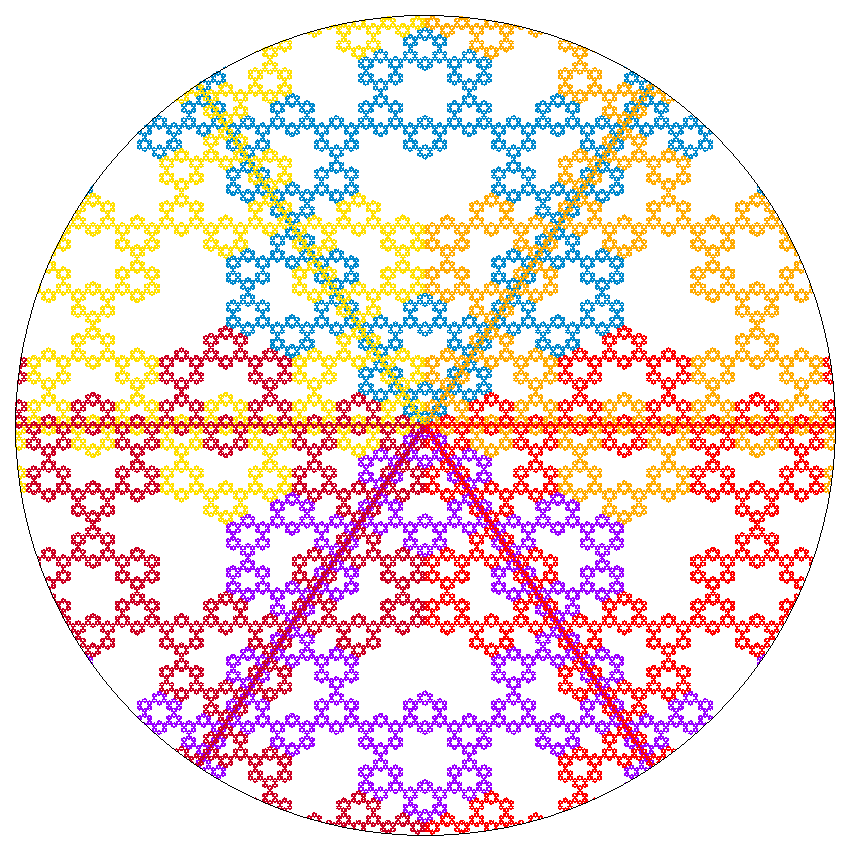}
   \end{center}
   \vskip20pt
   \caption{Macro-Fractal Snowflake with 6 fixed points.}{Cut-and-Zoom Image.}\label{macro-flake}
  \end{figure}
\vfill

\pagebreak
\phantom{m}
\vskip40pt

\begin{figure}[h!]
  \begin{center}
    \includegraphics[width=1\linewidth, keepaspectratio]{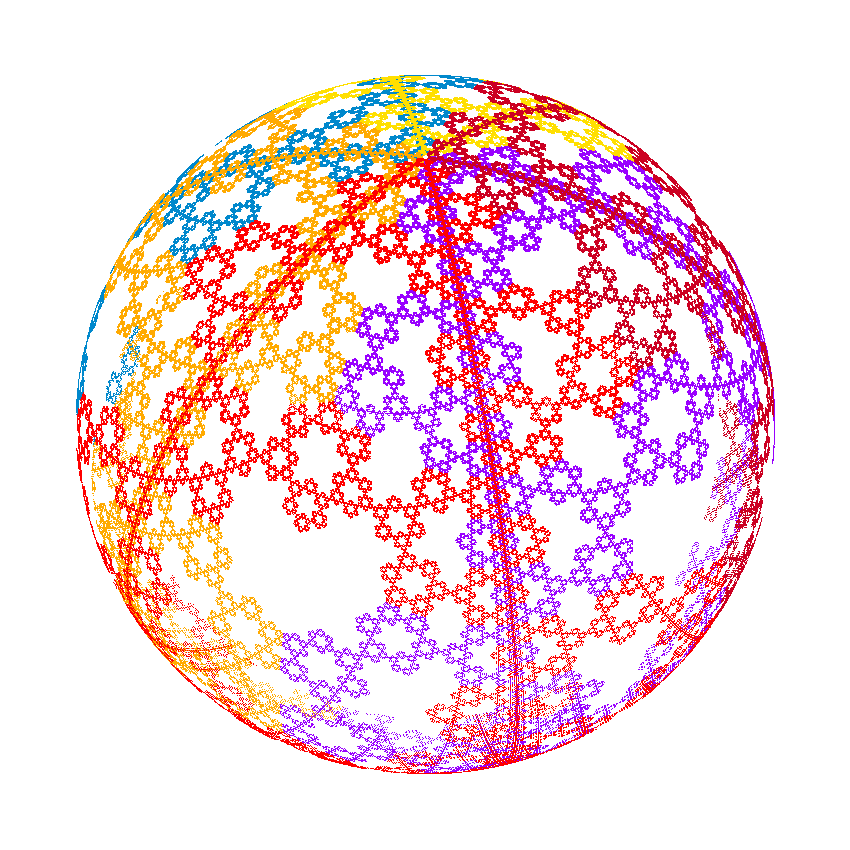}
   \end{center}
   \caption{Macro-Fractal  Snowflake with 6 fixed points.}{Spherical image.}\label{str-sn6}
 \end{figure}
\vfill

\pagebreak
\phantom{m}
\vskip40pt

\begin{figure}[h!]
  \begin{center}
    \includegraphics[width=0.9\linewidth, keepaspectratio]{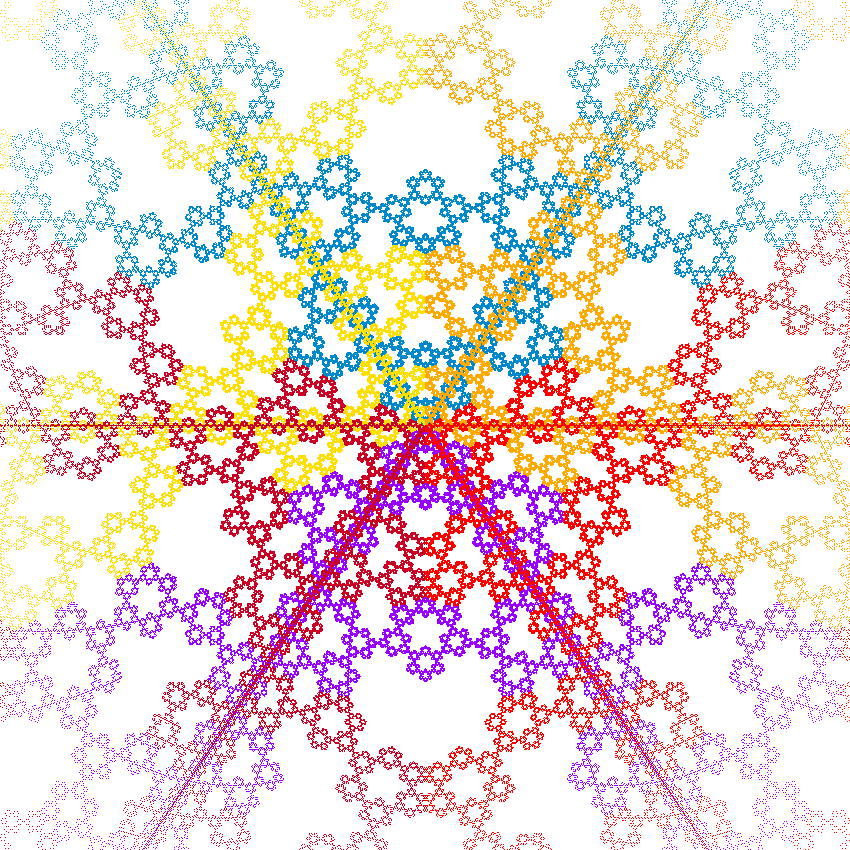}
   \end{center}
    \vskip20pt
   \caption{Macro-Fractal  Snowflake with 6 fixed points.}{Spherical image at a neighborhood of infnity.}\label{sph-sn6}
\end{figure}
\vfill

\pagebreak
\phantom{m}
\vskip40pt

\begin{figure}[h!]
  \begin{center}
    \includegraphics[width=0.95\linewidth, keepaspectratio]{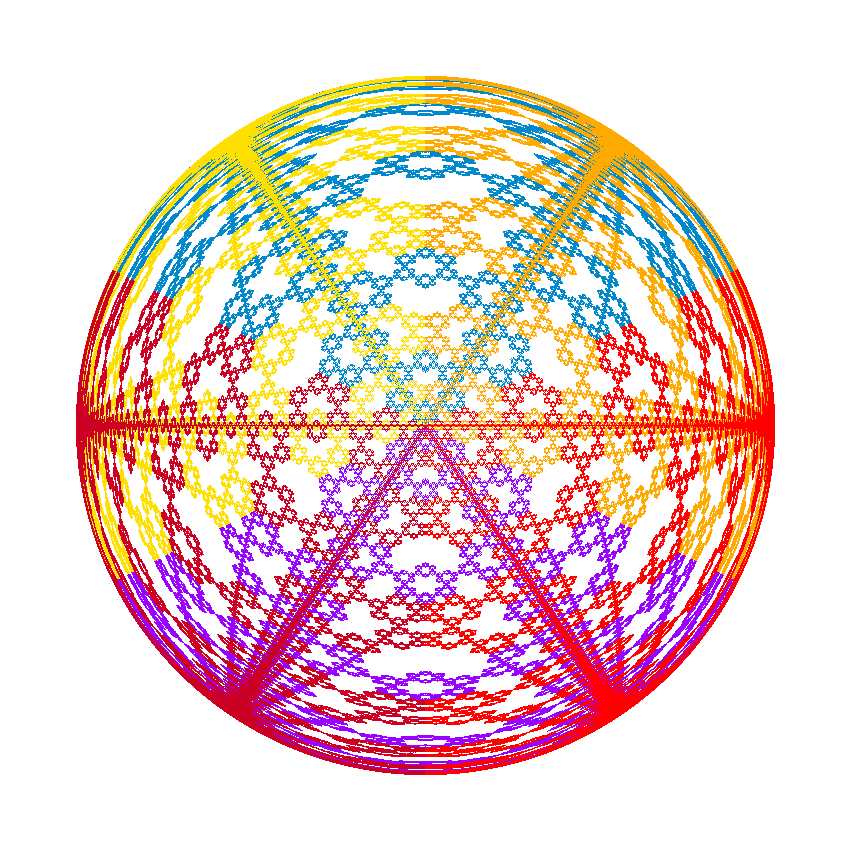}
   \end{center}
   \caption{Macro-Fractal  Snowflake with 6 fixed points.}{Semi-spherical image.}\label{semi-sn6}
\end{figure}
\vfill
\pagebreak
\phantom{m}
\vskip40pt

\begin{figure}[h!]
  \begin{center}
    \includegraphics[width=0.9\linewidth]{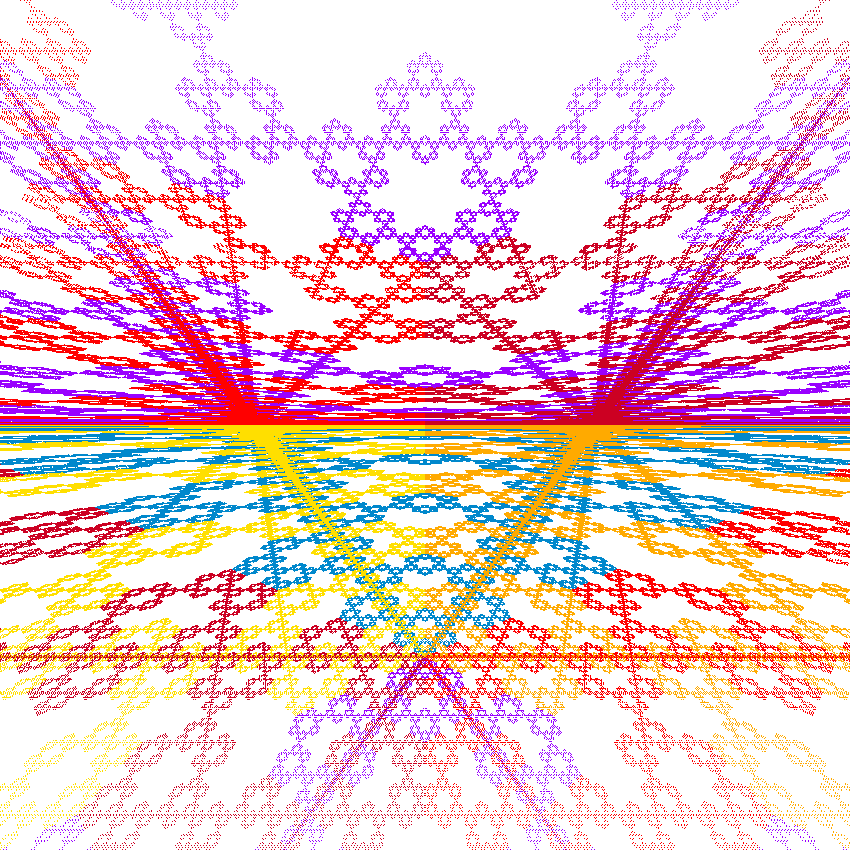}
 \end{center}  \vskip20pt
   \caption{Macro-Fractal  Snowflake with 6 fixed points.}{Projective image.}\label{proj-sn6}
\end{figure}
\vfill

\pagebreak

\subsection{Micro- and Macro-Fractal Snowflakes with 7 fixed points}
Let $$F_7=\{0\}\cup F_6=\{0\}\cup \{e^{i\pi k/3}:0\le k<6\}$$be the 7-element set consisting of the center and the vertices of a regular 6-gon on the complex plane. Consider the expanding multi-valued function
$$\Phi:\IC\setmap \IC,\;\;\Phi:z\mapsto 3z-2F_7$$whose inverse $$\Phi^{-1}:\IC\setmap\IC,\;\;\Phi^{-1}:z\mapsto \tfrac13z+\tfrac23F_7$$is a contracting multi-valued function. The dual fractals $\Fractal[\Phi]$ and $\Fractal[\Phi^{-1}]$ are called the {\em Macro-Fractal and Micro-Fractal Snowflakes with 7 fixed points}, respectively. Their images are drawn on Figures~\ref{micro-sn7}--\ref{proj-sn7}.
\vskip30pt

\begin{figure}[h!]
  \begin{center}
    \includegraphics[width=0.6\linewidth, height=0.68\linewidth]{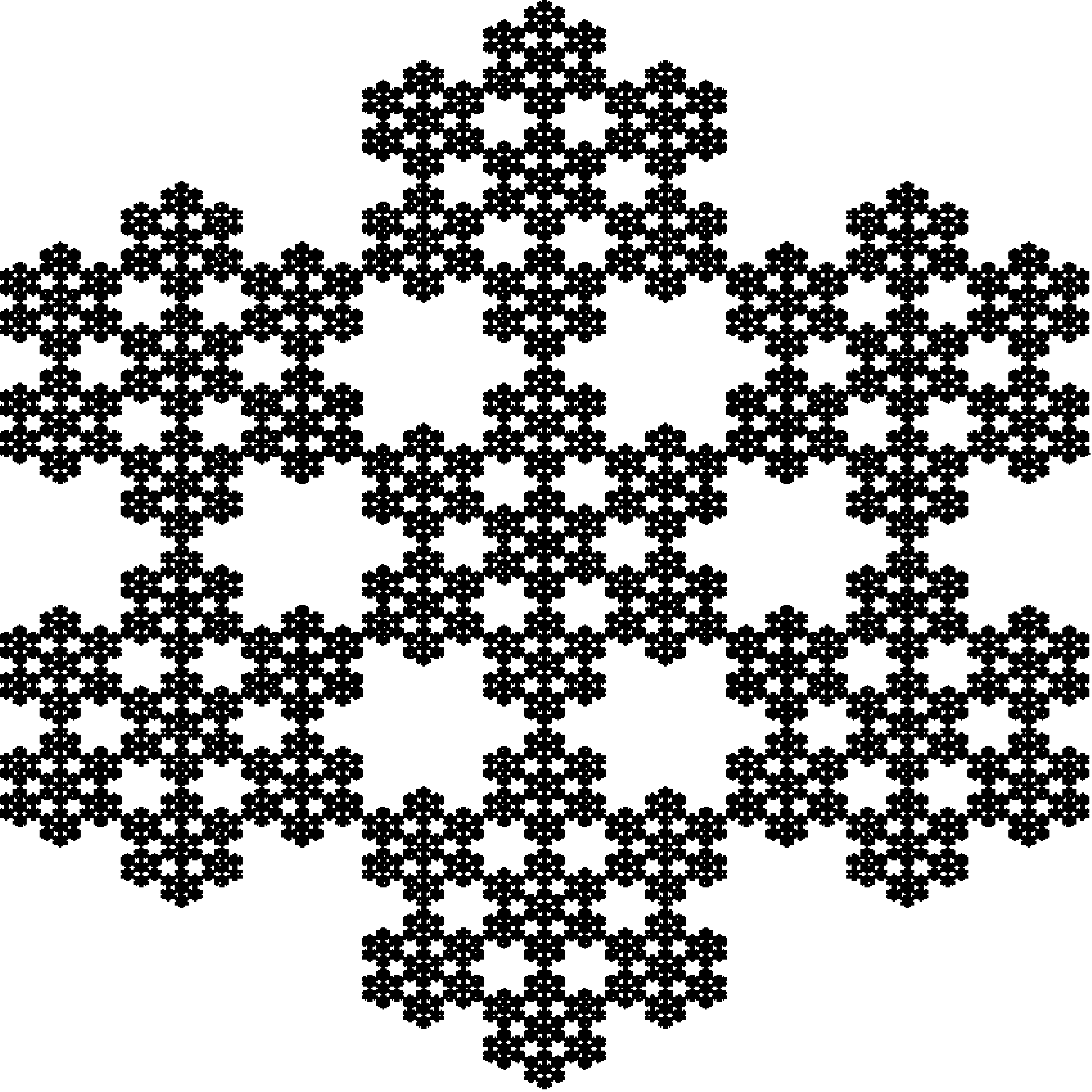}
   \end{center} \vskip20pt
   \caption{Micro-Fractal Snowflake with 7 fixed points.}\label{micro-sn7}
\end{figure}
\vfill
\pagebreak
\phantom{m}
\vskip40pt

\begin{figure}[h!]
  \begin{center}
    \includegraphics[width=0.9\linewidth, keepaspectratio]{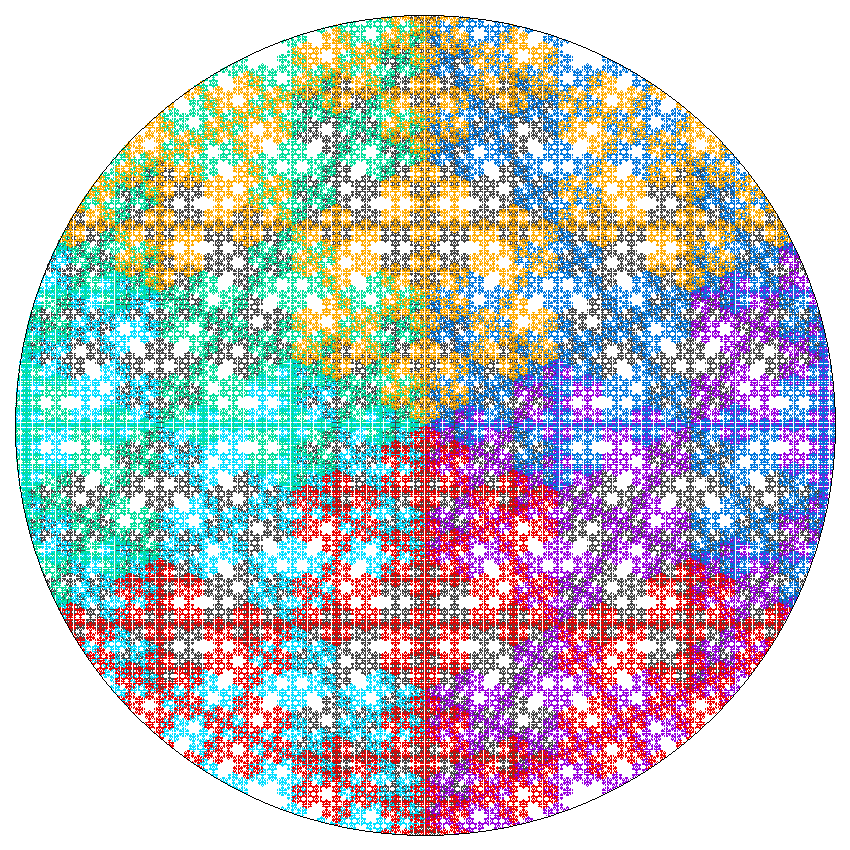}
   \end{center}
   \caption{Macro-Fractal Snowflake with 7 fixed points.}{Cut-and-Zoom image.}\label{macro-sn7}
    \end{figure}
\vfill
\pagebreak
\phantom{m}
\vskip40pt

\begin{figure}[h!]
  \begin{center}
    \includegraphics[width=0.9\linewidth, keepaspectratio]{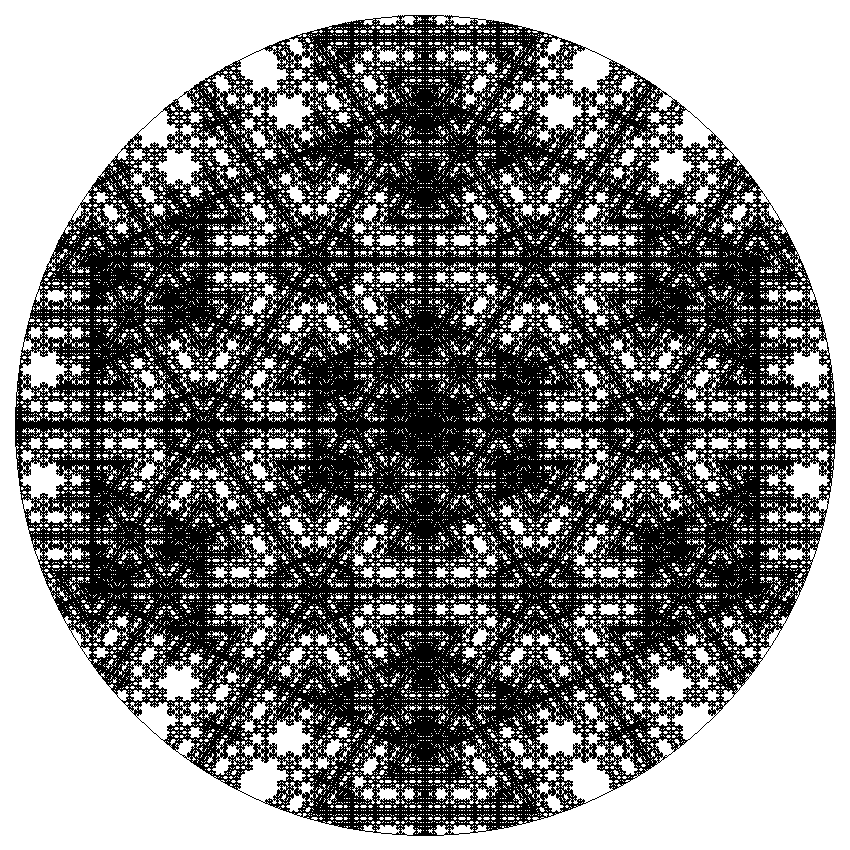}
   \end{center}\vskip20pt
   \caption{Micro-Fractal Snowflake with 7 fixed points.}{Interference effects on its black-and-white picture.}\label{macro-sn7-bw}
\end{figure}
\vfill
\pagebreak

\phantom{m}
\vskip40pt

\begin{figure}[h!]
  \begin{center}
    \includegraphics[width=1\linewidth, keepaspectratio]{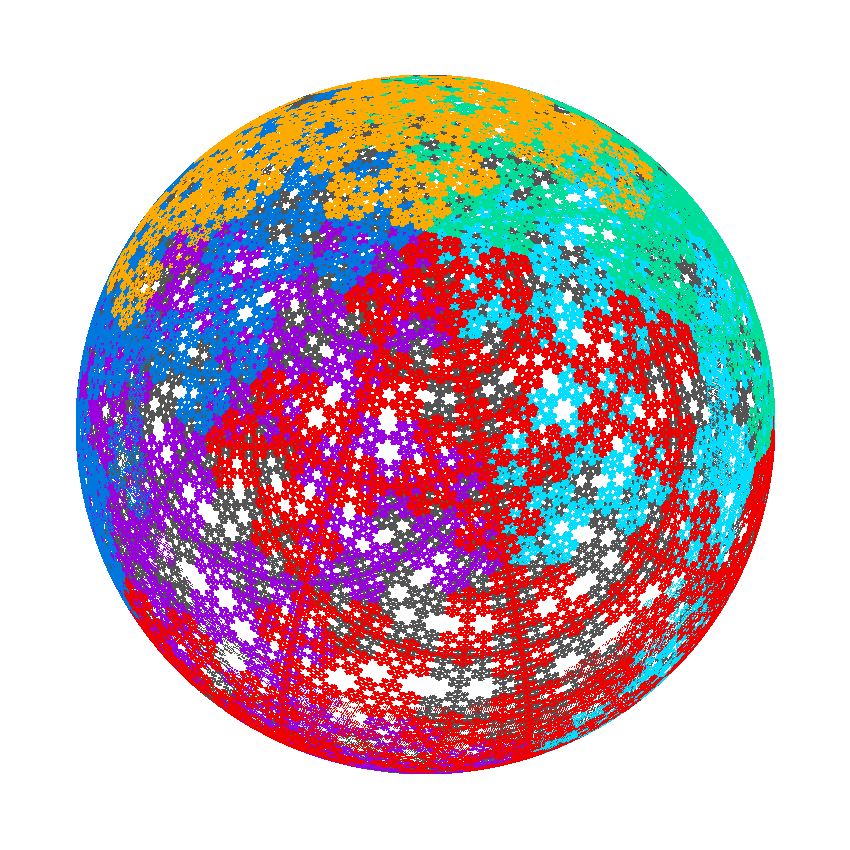}
   \end{center}
   \caption{Micro-Fractal Snowflake with 7 fixed points.}{Spherical image.}\label{str-sn7}
\end{figure}
\vfill
\pagebreak
\phantom{m}
\vskip40pt

\begin{figure}[h!]
  \begin{center}
    \includegraphics[width=0.9\linewidth, keepaspectratio]{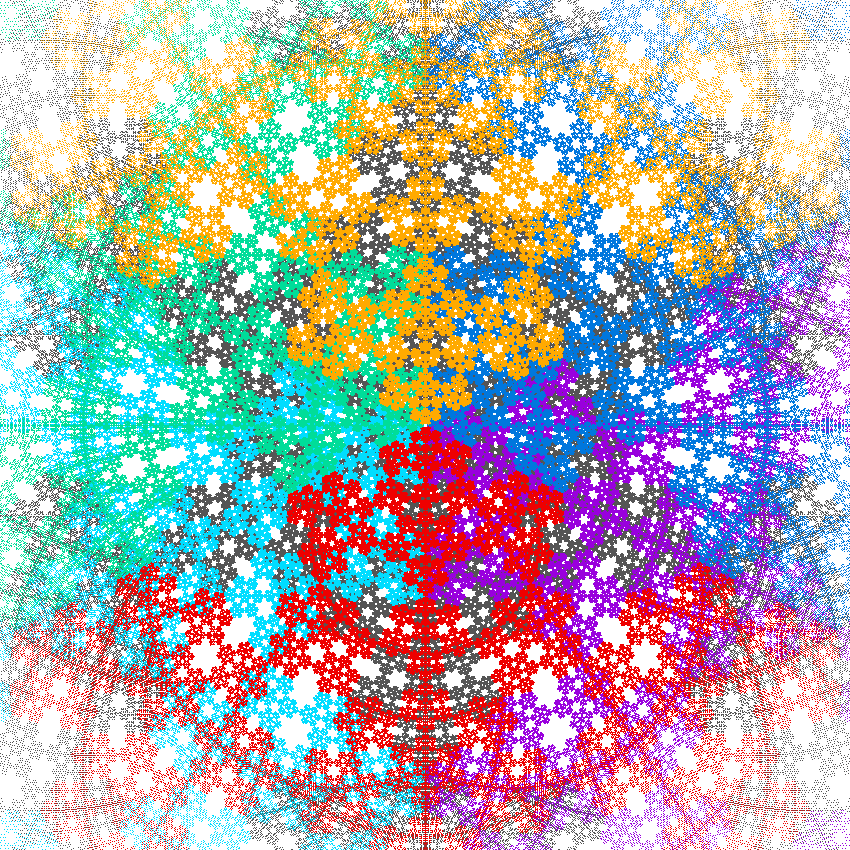}
   \end{center}\vskip20pt
   \caption{Micro-Fractal Snowflake with 7 fixed points.}{Spherical image at a neighborhood of infinity.}\label{sph-sn7}
\end{figure}
\vfill
\pagebreak
\phantom{m}
\vskip40pt

\begin{figure}[h!]
  \begin{center}
    \includegraphics[width=0.95\linewidth, keepaspectratio]{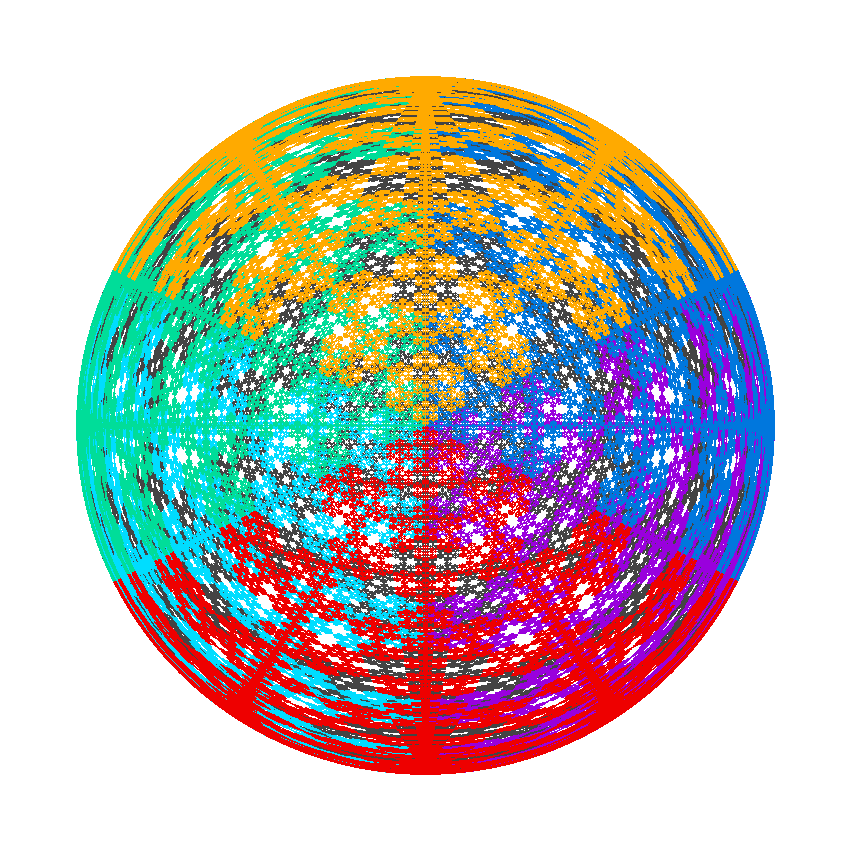}
   \end{center}
   \caption{Micro-Fractal Snowflake with 7 fixed points.}{Semi-spherical image.}\label{semi-sn7}
\end{figure}
\vfill
\pagebreak
\phantom{m}
\vskip40pt

\begin{figure}[h!]
  \begin{center}
    \includegraphics[width=0.9\linewidth]{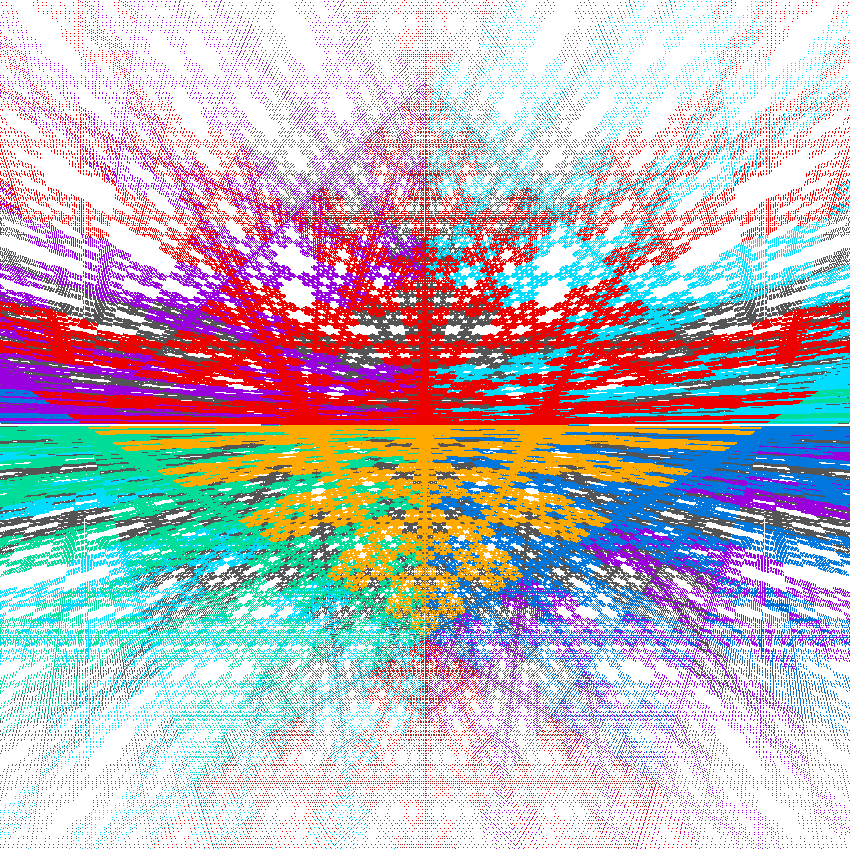}
   \end{center}\vskip20pt
   \caption{Micro-Fractal Snowflake with 7 fixed points.}{Projective image.}\label{proj-sn7}
\end{figure}
\vfill
\pagebreak
\phantom{m}
\vskip40pt

\subsection{Sierpi\'nski Micro- and Macro-Carpets}
Consider the 8-element subset $$F_8=\{-1,0,1\}^2\setminus\{(0,0)\}$$of the plane $\IR^2$ and two multi-valued functions
$$\Phi:\IC\setmap \IC,\;\;\Phi:z\mapsto 3z-2F_8$$whose inverse $$\Phi^{-1}:\IC\setmap\IC,\;\;\Phi^{-1}:z\mapsto \tfrac13z+\tfrac23F_8.$$ The dual fractals $\Fractal[\Phi]$ and $\Fractal[\Phi^{-1}]$ are called the {\em Sierpinski Micro-Carpet and Macro-Carpet}, respectively. Their images are drawn on Figures~\ref{micro-carpet}--\ref{proj-carpet}.
\vskip50pt

\begin{figure}[h!]
  \begin{center}
    \includegraphics[width=0.5\linewidth, keepaspectratio]{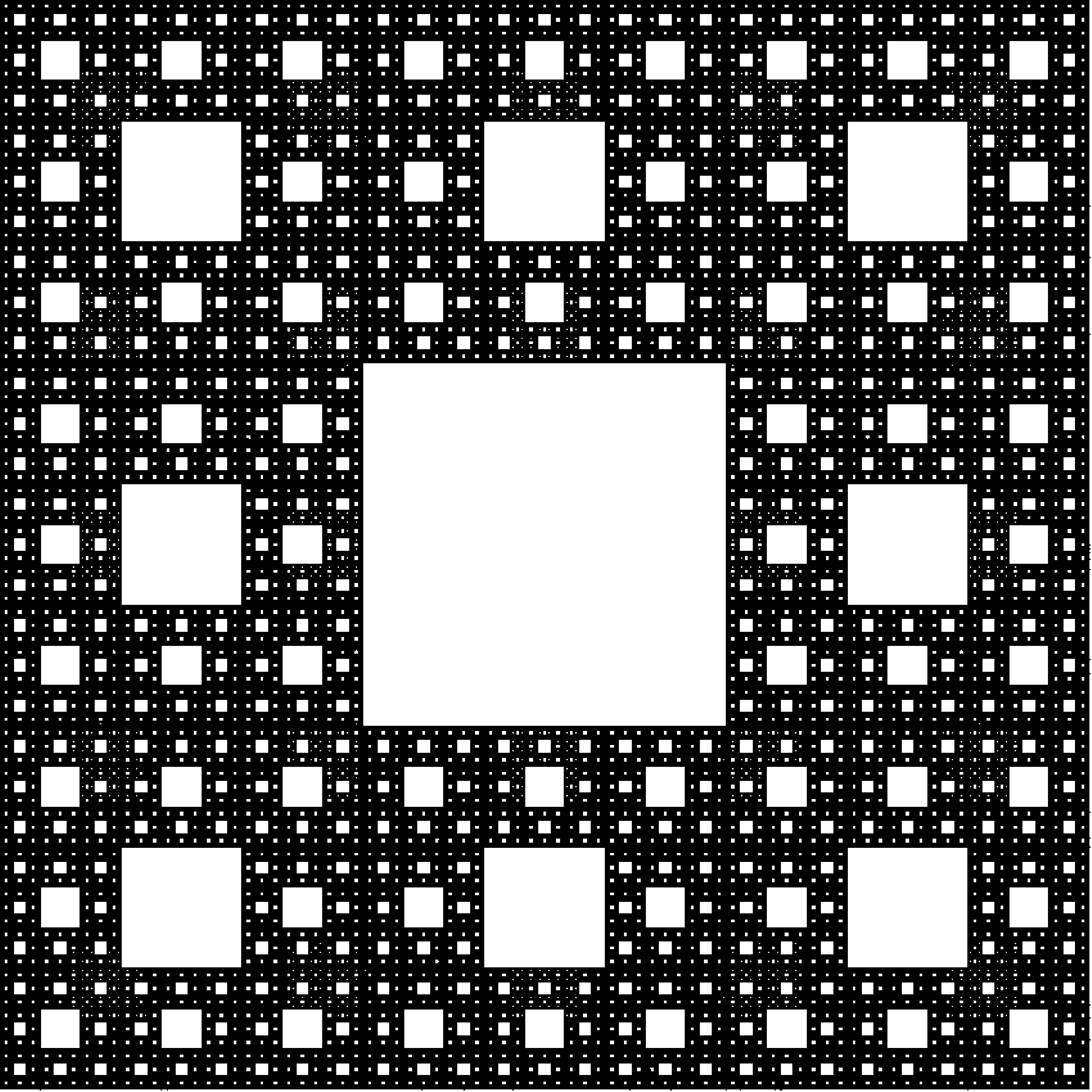}
    \vskip10pt
    \end{center}
   \caption{Sierpi\'nski Micro-Carpet.}\label{micro-carpet}
\end{figure}
\vfill
\pagebreak
\phantom{m}
\vskip40pt

\begin{figure}[h!]
  \begin{center}
    \includegraphics[width=0.9\linewidth, keepaspectratio]{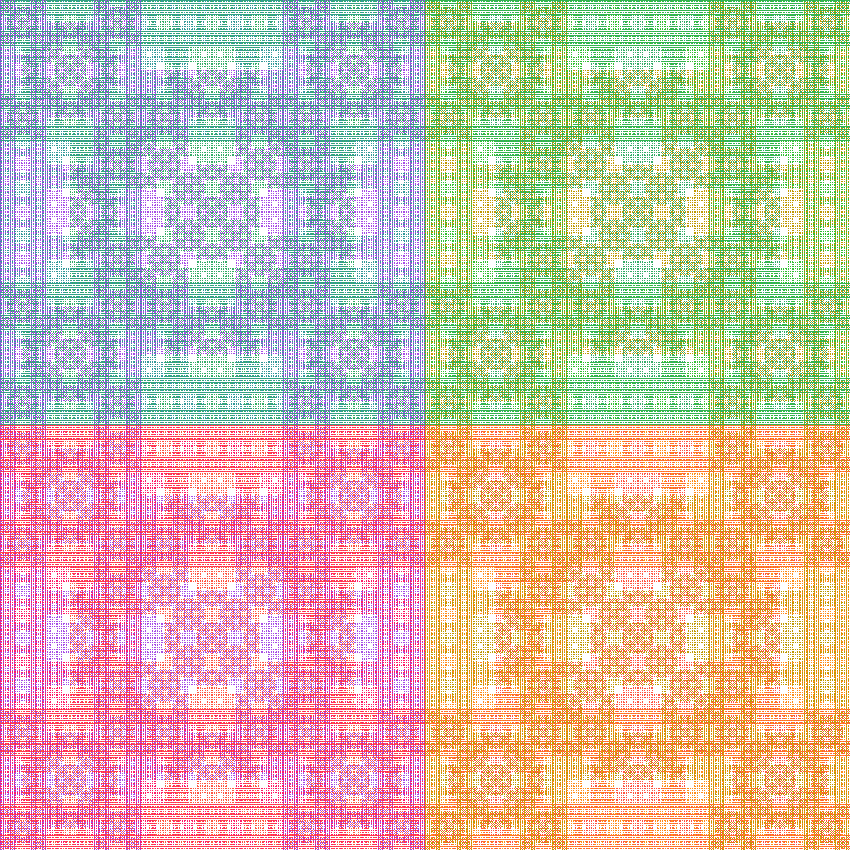} \end{center}\vskip20pt
   \caption{Sierpi\'nski Macro-Carpet.}{Cut-and-Zoom color image}\label{macro-carpet}
\end{figure}
\vfill
\pagebreak
\phantom{m}
\vskip40pt

\begin{figure}[h!]
  \begin{center}
    \includegraphics[width=0.9\linewidth, keepaspectratio]{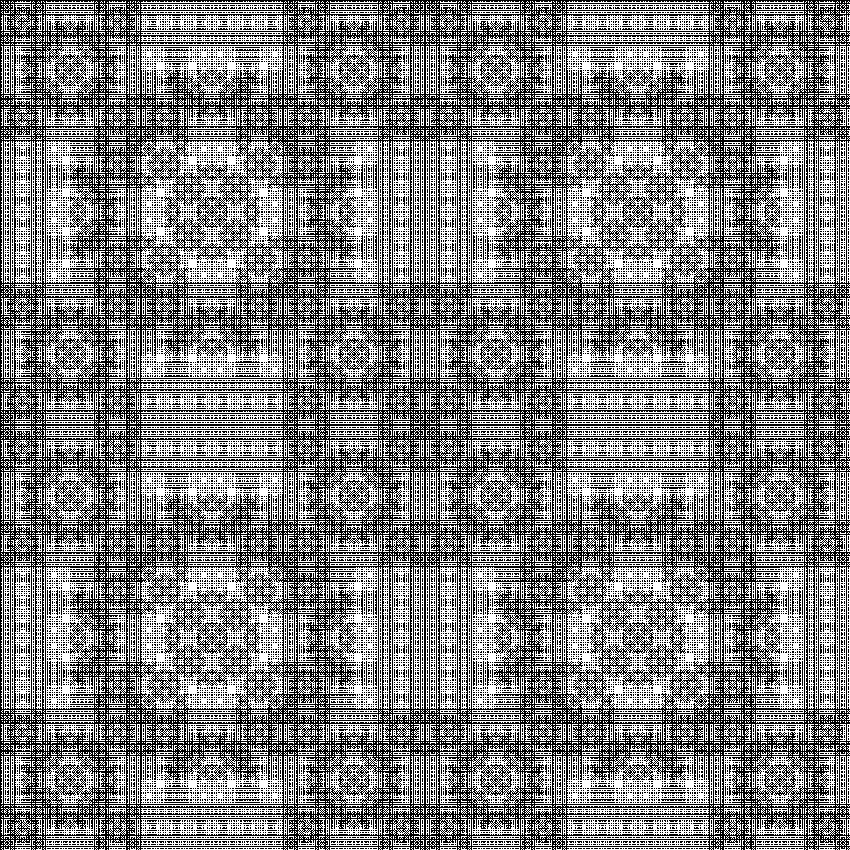}
  \end{center}\vskip20pt
   \caption{Sierpi\'nski Macro-Carpet.}{Interference effects on its black-and-white picture.}\label{macro-carpet-bw}
\end{figure}

\vfill
\pagebreak
\phantom{m}
\vskip40pt

\begin{figure}[h!]
  \begin{center}
    \includegraphics[width=1\linewidth, keepaspectratio]{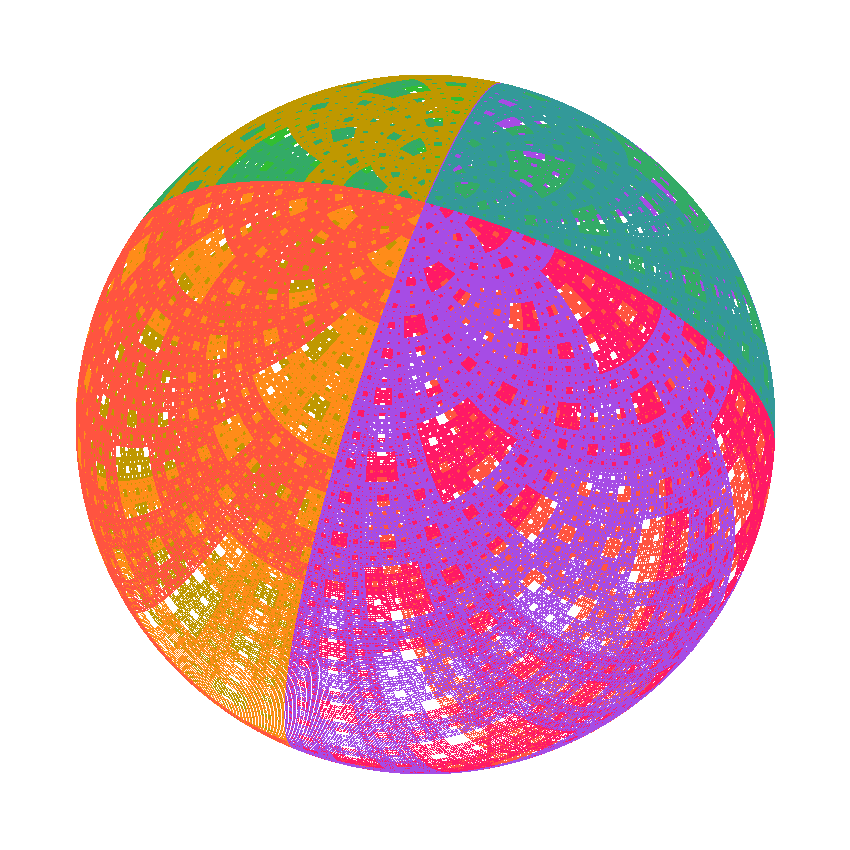}
   \end{center}
   \caption{Sierpi\'nski Macro-Carpet.}{Spherical image.}\label{str-carpet}
\end{figure}

\vfill
\pagebreak
\phantom{m}
\vskip40pt

\begin{figure}[h!]
  \begin{center}
    \includegraphics[width=0.9\linewidth, keepaspectratio]{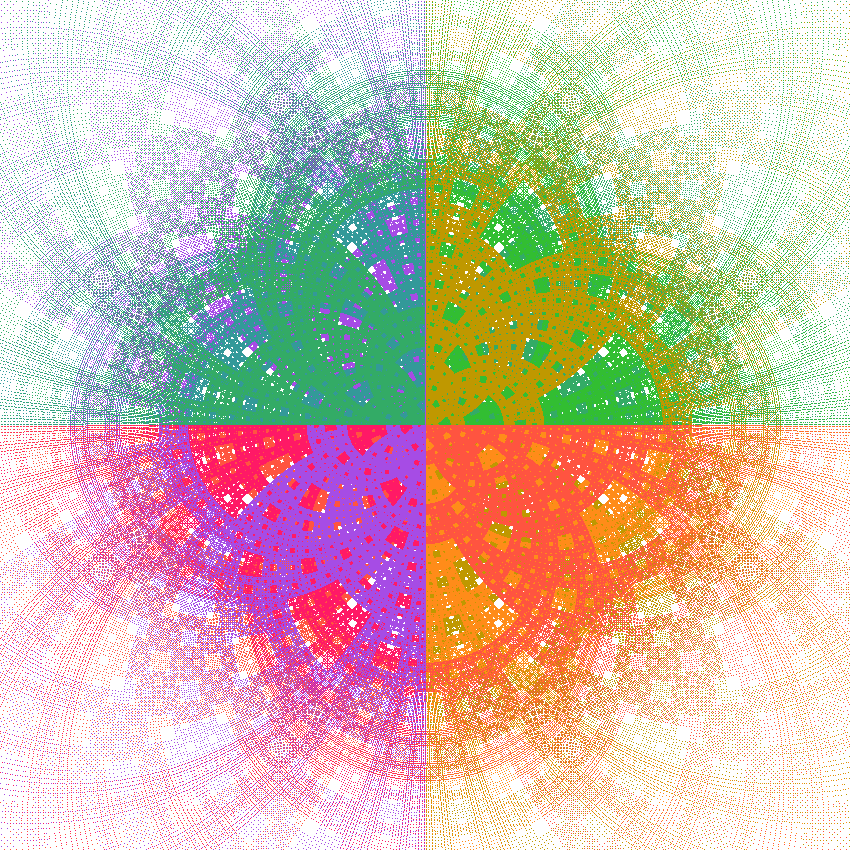}
\end{center}\vskip20pt
   \caption{Sierpi\'nski Macro-Carpet.}{Spherical image at neighborhood of infinity.}\label{sph-curp}
\end{figure}
\vfill
\pagebreak
\phantom{m}
\vskip40pt

\begin{figure}[h!]
  \begin{center}
    \includegraphics[width=0.95\linewidth, keepaspectratio]{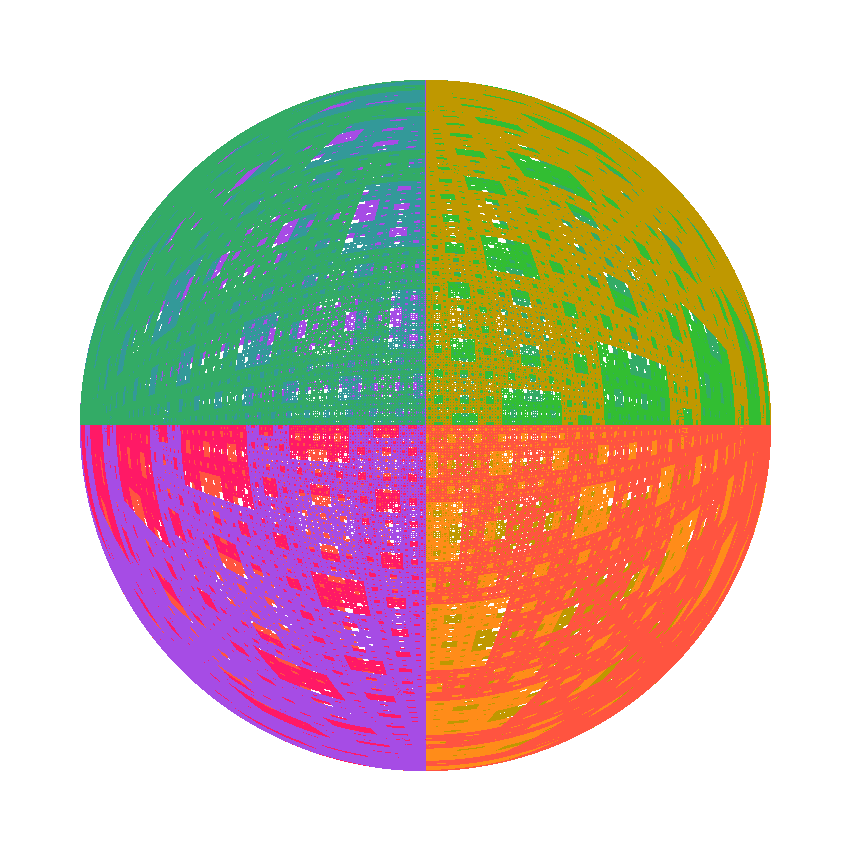}
   \end{center}
   \caption{Sierpi\'nski Macro-Carpet.}{Semi-spherical image.}\label{semi-serp}
\end{figure}
\vfill
\pagebreak
\phantom{m}
\vskip40pt

\begin{figure}[h!]
  \begin{center}
    \includegraphics[width=0.9\linewidth]{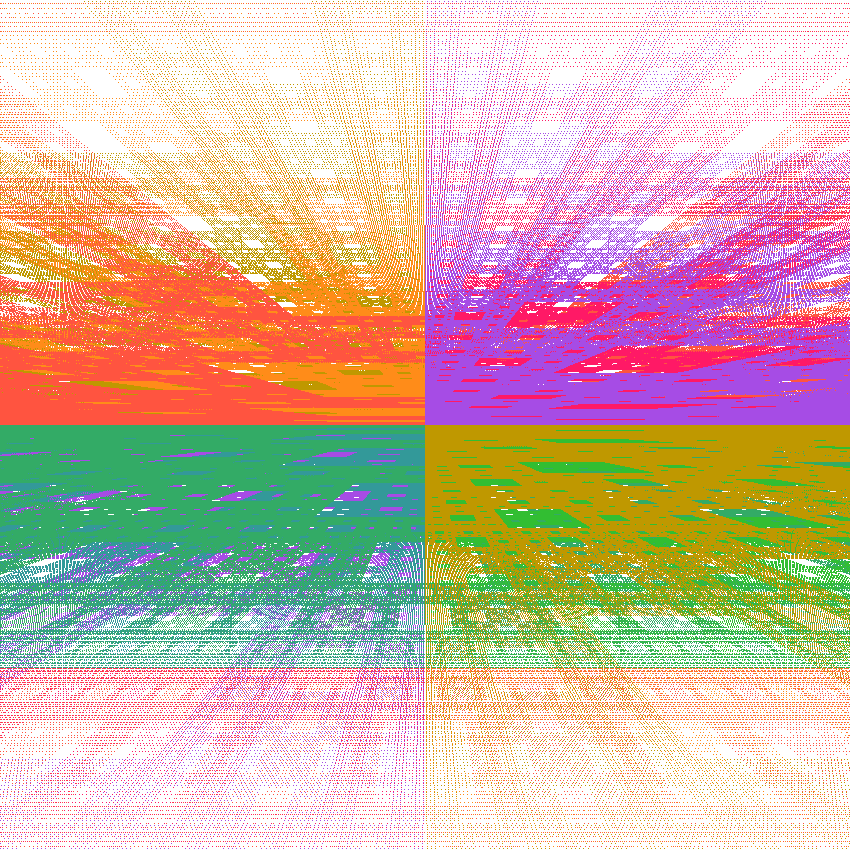}
    \end{center}\vskip20pt
   \caption{Sierpi\'nski Macro-Carpet.}{Projective image.}\label{proj-carpet}
\end{figure}
\end{document}